\documentclass[11pt,a4paper]{article}

\usepackage[utf8]{inputenc}
\usepackage[english]{babel}
\usepackage{amsmath}
\usepackage{amsfonts}
\usepackage{amssymb}
\usepackage{graphicx}
\usepackage[utf8]{inputenc}
\usepackage[english]{babel}
\usepackage{amsmath}
\usepackage{amsfonts}
\usepackage{amssymb}
\usepackage{amsthm}
\usepackage{mathtools}
\usepackage{nicefrac}
\usepackage[colorlinks=true,citecolor=blue]{hyperref}
\usepackage{cleveref}
\usepackage{bm}
\usepackage{float}
\usepackage{graphicx}
\usepackage{natbib}
\usepackage{bbm}
\usepackage{subfig}
\usepackage{wrapfig}
\usepackage[]{appendix}
\usepackage{caption}
\usepackage{setspace} 
\usepackage{animate}
\usepackage{algorithm}
\usepackage{algpseudocode}
\usepackage{pdfpages}
\usepackage[paper=a4paper,left=30mm,right=30mm,top=25.4mm,bottom=25.4mm]{geometry}

\newcommand{\easysum}[2]{\ensuremath{\underset{#1}{\overset{#2}{\sum}}}}
\newcommand{\es}[2]{\ensuremath{\underset{#1}{\overset{#2}{\sum}}}}
\newcount\colveccount
\newcommand*\colvec[1]{
        \global\colveccount#1
        \begin{pmatrix}
        \colvecnext
}
\def\colvecnext#1{
        #1
        \global\advance\colveccount-1
        \ifnum\colveccount>0
                \\
                \expandafter\colvecnext
        \else
                \end{pmatrix}
        \fi
}

\newtheorem{theorem}{Theorem}[section]
\newtheorem{lemma}[theorem]{Lemma}
\newtheorem{proposition}[theorem]{Proposition}
\newtheorem{corollary}[theorem]{Corollary}
\newtheorem{definition}[theorem]{Definition}
\newtheorem{remark}[theorem]{Remark}

\newcommand{\msrX}{\mathcal{M}_+(\mathcal{X})}
\newcommand{\msrY}{\mathcal{M}_+(\mathcal{Y})}
\newcommand{\X}{\mathcal{X}}
\newcommand{\Y}{\mathcal{Y}}
\newcommand{\BS}{\mathcal{B}}
\newcommand{\dum}{\mathfrak{d}}

\DeclareMathOperator*{\argmin}{arg\,min}
\DeclareMathOperator*{\ty}{\tilde{\mathcal{Y}}}
\DeclareMathOperator*{\esn}{\es{i=1}{N}}
\newcommand{\uamin}[1]{\ensuremath{\underset{#1}{\argmin}}} 
\definecolor{mygreen}{rgb}{0,0,0}
\definecolor{myred}{rgb}{0,0,0}

\newcommand{\footremember}[2]{%
	\footnote{#2}
	\newcounter{#1}
	\setcounter{#1}{\value{footnote}}%
}
\newcommand{\footrecall}[1]{%
	\footnotemark[\value{#1}]%
}

\author{Florian Heinemann \footremember{ims}{\scriptsize Institute for Mathematical
		Stochastics, University of G\"ottingen,
		Goldschmidtstra{\ss}e 7, 37077 G\"ottingen}
		\and
		Marcel Klatt \footrecall{ims}
		\and
	Axel Munk \footrecall{ims} \footnote{\scriptsize Max Planck Institute for Biophysical
		Chemistry, Am Fa{\ss}berg 11, 37077 G\"ottingen} \footnote{ \scriptsize University Medical Center Göttingen, Cluster of Excellence 2067 Multiscale Bioimaging - From molecular machines to networks of excitable cells}}

	\title{Kantorovich-Rubinstein distance and barycenter for finitely supported measures: Foundations and Algorithms}

\graphicspath{{./graphics/}}
\begin{document}

\maketitle

\begin{abstract}
    The purpose of this paper is to provide a systematic discussion of a generalized barycenter based on a variant of unbalanced optimal transport (UOT) that defines a distance between general non-negative, finitely supported measures by allowing for mass creation and destruction modeled by some cost parameter. They are denoted as Kantorovich-Rubinstein (KR) barycenter and distance. In particular, we detail the influence of the cost parameter to structural properties of the KR barycenter and the KR distance. For the latter we highlight a closed form solution on ultra-metric trees. The support of such KR barycenters of finitely supported measures turns out to be finite in general and its structure to be explicitly specified by the support of the input measures. Additionally, we prove the existence of sparse KR barycenters and discuss potential computational approaches. The performance of the KR barycenter is compared to the OT barycenter on a multitude of synthetic datasets. We also consider barycenters based on the recently introduced Gaussian Hellinger-Kantorovich and Wasserstein-Fisher-Rao distances.    
\end{abstract}

\section{Introduction}
Over the past decade, optimal transport (OT) based concepts for data a\-na\-ly\-sis \citep[for a thorough treatment of the mathematical foundations of optimal transport see e.g.][]{rachev1998mass,villani2008optimal,santambrogio2015optimal} have seen increasing popularity. This is mainly due to the fact that OT based methods respect important features of the data's geometric structure. Furthermore, noteworthy advances have been achieved in various areas, such as optimisation \citep{bertsimas1997introduction,wolsey1999integer,grotschel2012geometric}, machine learning \citep{frogner2015learning,peyre2019computational,xie2020fast}, computer vision \citep{gangbo2000shape,su2015shape,solomon2015convolutional} and statistical inference \citep{sommerfeld2018inference,panaretos2020invitation,hallin2021multivariate}, among others. This methodological and computational progress recently also paved the way to novel areas of applications including genetics \citep{evans2012phylogenetic,schiebinger2019optimal} and cell biology \citep{gellert2019substrate,klatt2020empirical,tameling2021colocalization,wang2021revisiting}, to cite but a few. Of particular importance from a data analysis point of view are extensions to compare more than two measures, a prominent proposal being the Fr\'echet mean \citep{frechet1948elements}, in the present context known as \emph{Wasserstein barycenter} \citep{agueh2011barycenters}. Wasserstein barycenters allow for a notion of average on the space of probability measures, which is well-adapted to the geometry of the data \citep{alvarez2016fixed,anderes2016discrete}. With recent progress on their computation \citep{cuturi2014fast,carlier2015numerical,bonneel2015sliced,kroshnin2019complexity,ge2019interior,heinemann2022randomized} they establish themselves even further as a promising tool in many fields of data analysis, such as texture mixing \citep{rabin2011wasserstein}, distributional clustering \citep{ye2017fast}, histogram regression \citep{bonneel2016wasserstein}, domain adaptation \citep{montesuma2021wasserstein} and unsupervised learning \citep{schmitz2018wasserstein}, among others.\\
However, a well known drawback of the Wasserstein distance and its barycenters in various applications is their limitation to measures with equal total mass. In fact, in many real world instances the difference in total mass intensity is of crucial importance. Employing vanilla Wasserstein based tools on general positive measures necessitates the usage of a normalisation procedure to enforce mass equality between the measures. This approach is, by design, oblivious to the mass differences between the original measures and can limit its use in applications. Exemplary, we mention that normalisation destroys stoichiometric features in the analysis of protein interaction and pathways as pointed out in \cite{tameling2021colocalization}. Overall, this might lead to incorrect conclusions on specific applications. An illustrative example is given in \Cref{fig:splitmatch}.

\begin{figure}
\centering
\includegraphics[width=\textwidth]{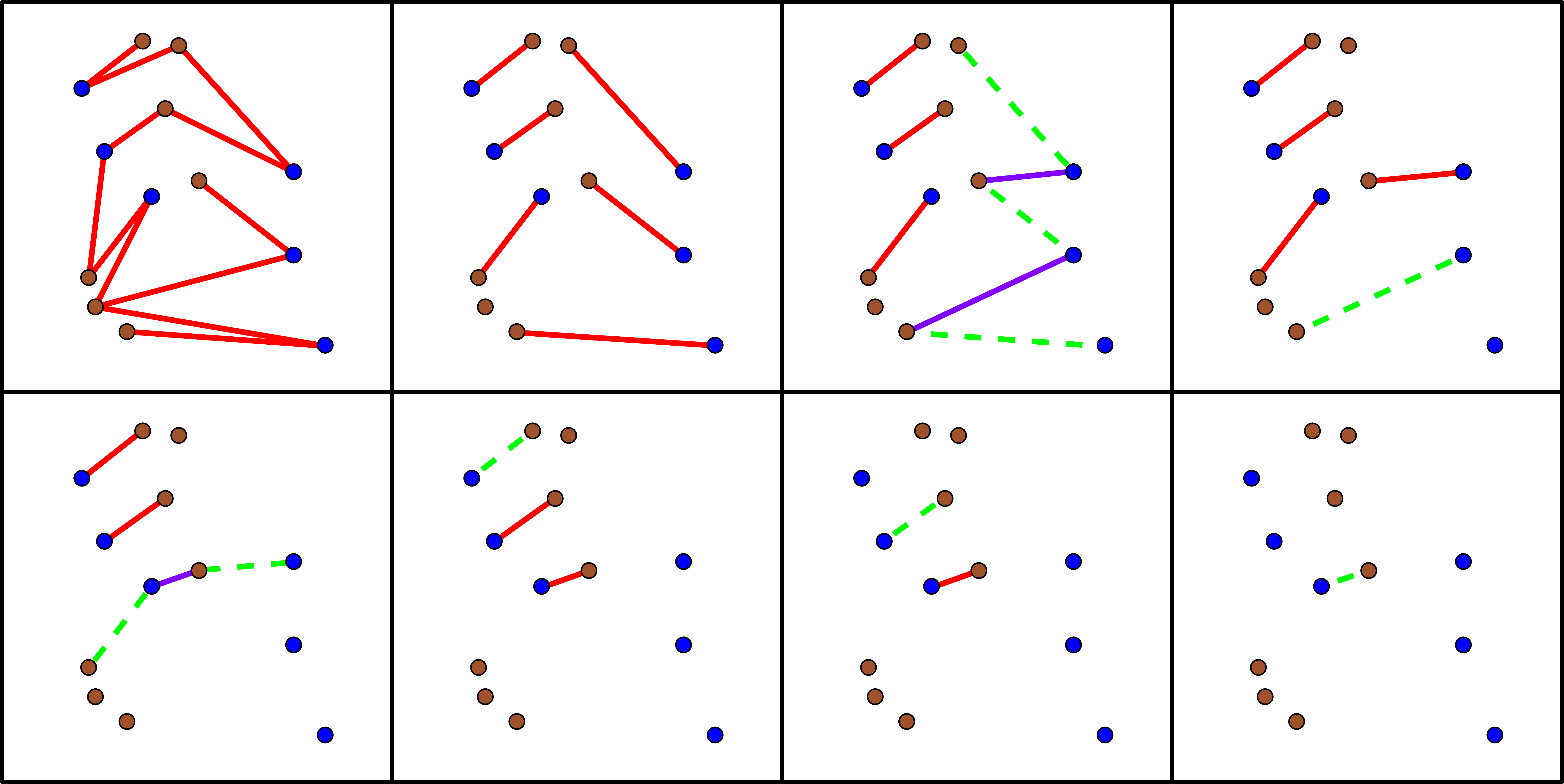}
  \caption{(Unbalanced) OT between two measures (support in blue and brown, respectively) with weights equal to one at each support point. \textbf{Top-Left:} OT plan (red) between normalised versions of the two measures. \textbf{Rest:} UOT plans (red/purple) between non-normalised measures. From top-left to bottom-right $C$ is decreasing. The edges, which have been removed most recently due to the reduction of $C$, are shown in green. Edges which have been added to the UOT graph due to the most recent reduction of $C$ are marked in purple.} \label{fig:splitmatch}
\end{figure}

\subsection{Prior Work}
The limitation of OT based concepts dealing only with measures of equal total mass has opened a wealth of approaches to account for more general measures. As an early proposal of this idea, the \emph{partial OT formulation} \citep{caffarelli2010free,figalli2010optimal} suggests to fix the total mass of the OT plan in advance, while relaxing the marginal constraints. Comparably more recent are \emph{entropy transport formulations}\footnote{Critically, this is not be confused with entropy \emph{regularized} optimal transport, which is a popular computational approach adding an entropy penalty term to the OT problem to allow for efficient, approximate computations \citep{cuturi2013sinkhorn,benamou2015iterative,carlier2017convergence}}. This general framework removes the marginal constraints and instead uses a divergence functional to measure the deviation between the transport marginals and the input measures. The entropy transport framework encompasses the \textit{Hellinger-Kantorovich} distance \citep{liero2018optimal,chizat2018scaling}, also known as \textit{Wasserstein-Fisher-Rao} distance \citep{chizat2018interpolating} and the \textit{Gaussian Hellinger-Kantorovich} distance \citep{liero2018optimal}. Inherent to all of these models is their dependency on parameters whose exact influence on the models' properties is generally not well understood. An alternative idea is based on extending the well-studied dynamic formulation of OT \citep{benamou2000computational} to measures with different total masses. With a focus on its geodesic properties, this approach has been studied in several works \citep{chizat2018interpolating,chizat2018unbalanced,gangbo2019unnormalized}. \\
In this paper, we rely on a simple and intuitive idea based on the seminal work of \cite{kantorovich1958space}. This accounts for mass construction and deletion at a cost modeled by some prespecified parameter \citep[for details see also][]{hanin1992kantorovich,guittet2002extended}. It leads to the \emph{Kantorovich-Rubinstein distance (KRD)} which curiously has been revisited several times under different names by various authors. For $p=1$, it has been referred to as Earth Mover's Distance \citep{pele2008linear}, and generalized Wasserstein distance \citep{piccoli2014generalized}, while for general $p\geq 1$ common terminology includes Kantorovich distance \citep{gramfort2015fast}, generalized KRD \citep{sato2020fast}, transport-trans\-form metric \citep{muller2020metrics} and robust optimal transport distance \citep{mukherjee2021outlier}.\\

\subsection{Contributions}
In this work, we define barycenters with respect to the KRD and investigate their fundamental properties from a data analysis point of view. This extends the popular notion of Wasserstein barycenters to unbalanced barycenters (UBCs), i.e., barycenters of measures of different total masses. Similary, UBCs have been considered explicitly for the Hellinger-Kantorovich distance \citep{chung2020barycenters,friesecke2021barycenters} and for the partial OT distance for absolutely continuous measures \citep{kitagawa2015multi}. Notably, the well-known approach of matrix scaling algorithms has been shown to provide a general framework to approximate any UBC based on entropy optimal transport \citep{chizat2018scaling} of finitely supported measures. Closely related to our approach is the work by \cite{muller2020metrics} approximating the KR barycenter in the special case of point patterns.\\
\textbf{The KR distance:} Let $(\X,d)$ be a finite metric space, where $\X =\{ x_1,\dots ,x_N \}$ and
\begin{align*}
\msrX \coloneqq \left\lbrace \mu \in \mathbb{R}^{\vert\X\vert} \,\mid\, \mu(x)\geq 0\, \forall x \in \X \right\rbrace
\end{align*}
is the set of non-negative measures\footnote{A non-negative measure on a finite space $\X$ is uniquely characterized by the values it assigns to each singleton $\{ x\}$. To ease notation we write $\mu(x)$ instead of $\mu(\{x\})$. The corresponding $\sigma$-field is always to be understood as the powerset of $\mathcal{X}$.} on $\X$. For a measure $\mu\in\msrX$ its total mass is defined as $\mathbb{M}(\mu)\coloneqq \sum_{x\in\X}\mu(x)$ and the subset of non-negative measures with total mass equal to one is the set of probability measures $\mathcal{P}(\X)$. If $\pi\in \mathcal{M}_+(\X\times\X)$ is a measure on the product space $\X \times \X$ its marginals are defined as $\pi(x,\X)\coloneqq \sum_{x^\prime\in\X} \pi(x,x^\prime)$ and $\pi(\X,x^\prime)\coloneqq \sum_{x \in \X}\pi(x,x^\prime)$, respectively. For two measures $\mu,\nu\in\msrX$ we define the set of \emph{non-negative sub-couplings} as
\begin{align}\label{eq:subcouplings}
\begin{split}
    \Pi_{\leq}(\mu,\nu)\coloneqq \lbrace \pi\in \mathcal{M}_+(\X\times\X) \mid \,&\pi(x,\X) \leq  \mu(x),\, \\ &\pi(\X,x^\prime)  \leq  \nu(x^\prime)\, \forall \, x,x^\prime\in\mathcal{X}  \rbrace.
\end{split}
\end{align}
Similarly, we denote the set of \emph{couplings} between $\mu$ and $\nu$ as $\Pi_=(\mu,\nu)$, where the inequality constraints in (\ref{eq:subcouplings}) are replaced by equalities.
 For $p\geq 1$ and a parameter $C>0$, \emph{unbalanced optimal transport} (UOT) between two measures $\mu,\nu\in\msrX$ is defined as
\begin{align}\label{eq:UOT}
\begin{split}
\text{UOT}_{p,C}(\mu,\nu)\coloneqq 
\min_{\pi\in \Pi_{\leq}(\mu,\nu)} &\sum\limits_{x,x^\prime \in \X}d^p(x,x^\prime)\pi(x,x^\prime)\\&+C^p\left(\frac{\mathbb{M}(\mu)+\mathbb{M}(\nu)}{2}-\mathbb{M}(\pi)\right). 
\end{split}
\end{align}
\begin{figure}
\centering
    \includegraphics[width=\textwidth]{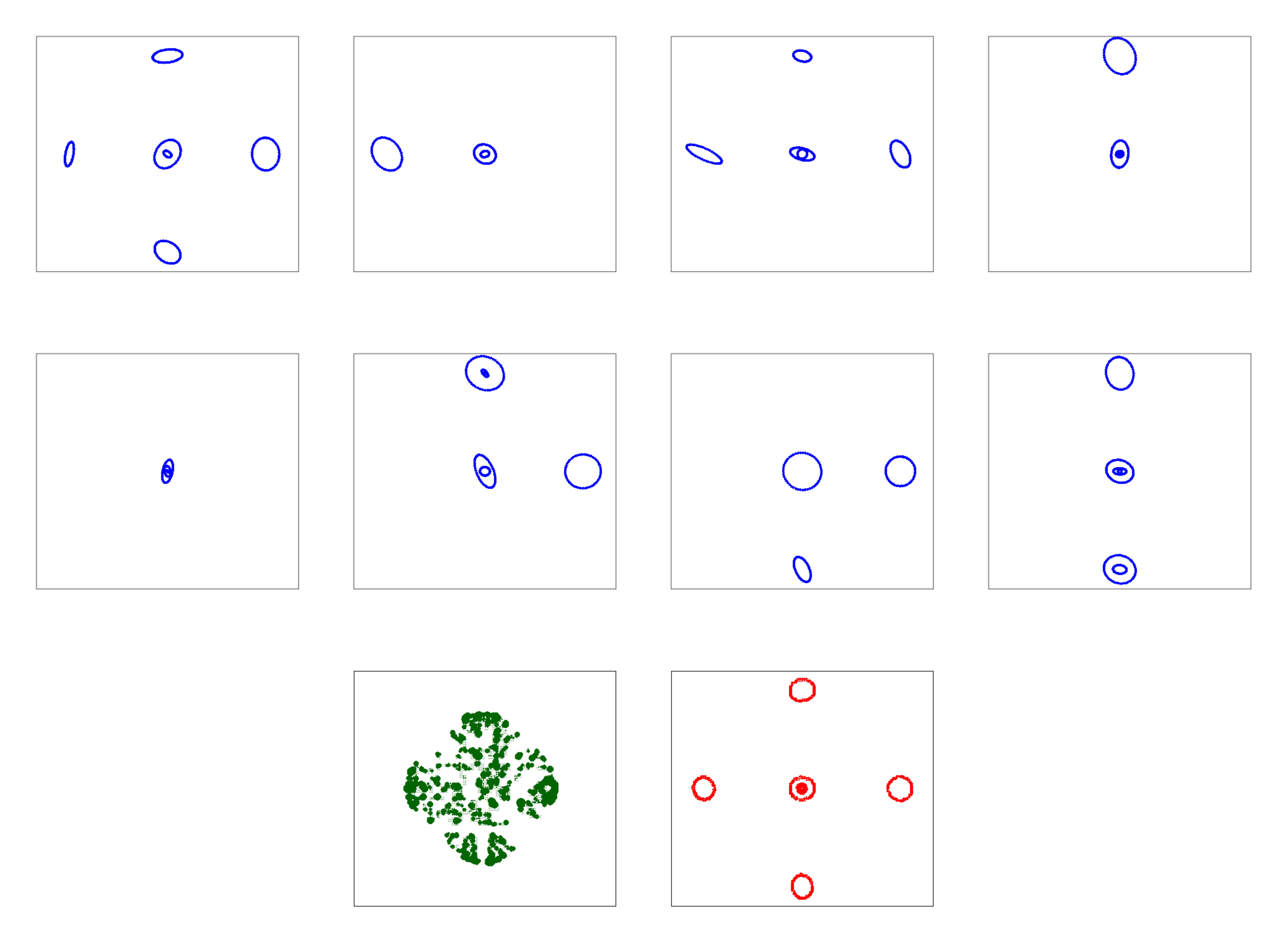}
\caption{\textbf{Upper two rows:} An excerpt of eight instances of a dataset of $N=100$ nested ellipses at up to $5$ different clusters in $[0,1]^2$. The number of ellipses in each cluster follows a Poisson distribution. For the cluster in the center the intensity is $2$ and for the four outer clusters the intensity is $1$. Each ellipse is discretized into $50$ points with mass $1$ at each location. \textcolor{myred}{For details on the computational methods refer to \Cref{sec:ComputationsSimulations}}. \textbf{Bottom-Left:} The Wasserstein barycenter of the normalized versions of these measures \textcolor{myred}{(runtime about 15 hours)}. \textbf{Bottom-Right:} The $(2,0.2)$-barycenter of these measures \textcolor{myred}{(runtime about 30 minutes)}. The $(2,C)$-barycenter for different values of $C$ can be seen in \Cref{fig:cluster_ellipses_ccurve}. \label{fig:cluster_ellipses}}
\end{figure}

Notably, $\text{UOT}_{p,C}(\mu,\nu)$ is finite for all measures $\mu,\nu\in\msrX$ with possibly different total masses and a solution of (\ref{eq:UOT}) always exists. Here, the parameter $C$ penalizes deviation of mass from the marginals of $\pi$ with respect to the input measures $\mu,\nu\in\msrX$. In particular and unlike the (balanced) OT problem
\begin{align*}
    \text{OT}_p(\mu,\nu)\coloneqq\min_{\Pi_{=}(\mu,\nu)} \sum\limits_{x,x^\prime \in \X}d^p(x,x^\prime)\pi(x,x^\prime) 
\end{align*}
defined only for measures $\mu,\nu\in\msrX$ with equal total mass $\mathbb{M}(\mu)=\mathbb{M}(\nu)$, UOT in \eqref{eq:UOT} relaxes the marginal constraint and allows optimal solutions to have more flexible marginals. Based upon UOT we define the $p$-th order \emph{Kanto\-rovich-Rubinstein distance} between two measures $\mu,\nu\in\msrX$ as
\begin{align}\label{eq:KRdistance}
\text{KR}_{p,C}(\mu,\nu)\coloneqq \left( \text{UOT}_{p,C}(\mu,\nu)\right)^{\nicefrac{1}{p}}.
\end{align}
For any $p\geq 1$, it defines a distance on the space of non-negative measures $\msrX$ and it is an extension of the well-known \emph{$p$-Wasserstein distance} $W_p(\mu,\nu)\coloneqq \left(\text{OT}_p(\mu,\nu)\right)^{\nicefrac{1}{p}}$ defined only for measures of equal total mass. Indeed, the KRD is shown to interpolate in-between \emph{OT on small scales} and \emph{point-wise comparisons on large scales} (\Cref{prop:KRdistance}) relative to the parameter $C$. This allows for an intuitive interpretation of the KRD. More precisely, in \Cref{lem:transportC}, we detail a clear geometrical connection between the value of $C$ and the structure of the UOT. In particular, this contrasts the closely related partial OT problem \citep{figalli2010optimal} mentioned above. Employing Lagrange multipliers one can see that for any choice of $C$, there exists a fixed mass $m$ of the partial OT problem, such that these two problems are equivalent. However, finding this value of $m$ requires to solve the UOT problem. We stress that the influence of $m$ on the resulting transport is in general hard to determine, while the impact of $C$ is intuitively clear. Thus, this perspective seems better suited to many applications.
For the specific case of measures supported on ultrametric trees (\Cref{subsec:trees}) we prove (\Cref{thm:KRultrametric}) an analogue of the well-known closed formula for the $p$-Wasserstein distance \citep{kloeckner2015geometric}. Additionally, the computation of the KRD is known to be equivalent to solving a related balanced OT problem \citep{guittet2002extended}, allowing to apply any state-of-the-art solver with minimal modifications to compute the KRD and plan.\\

\textbf{The KR barycenter:} The KRD also lends itself to define a notion of a barycenter for a collection of measures as a generalization of the $p$-Wasserstein barycenter defined for probability measures $\mu_1,\ldots,\mu_J\in\mathcal{P}(\X)$ as 
 \begin{align}\label{eq:otbarycenter}
\tilde{\mu}\in \argmin_{\mu\in\mathcal{P}(\Y)} \frac{1}{J}\sum_{i=1}^J W_p^p(\mu,\mu_i).
\end{align} 
Here, $(\X,d)$ is assumed to be embedded in some ambient space $(\Y,d)$, e.g., an Euclidean space with $\X \subset \Y$. The distance $d$ on $\X$ is understood to be the distance on $\Y$ restricted to $\X$. For $\mu_1,\dots,\mu_J\in \msrX$, any measure
\begin{equation}\label{eq:KRbarycenter}
    \mu^\star\in \argmin_{\mu\in\mathcal{M}_+(\Y)} F_{p,C}(\mu)\coloneqq\frac{1}{J}\easysum{i=1}{J}\text{KR}_{p,C}^p(\mu_i,\mu)
\end{equation}
is said to be a \emph{$(p,C)$-Kantorovich-Rubinstein barycenter} or \emph{$(p,C)$-barycenter} for short\footnote{For the sake of readability, the weights in this definition are fixed to $1/J$, though it is easy to adapt all instances of their occurrence in this work to arbitrary positive weights $\lambda_1,\dots ,\lambda_J$, summing to $1$.}. We refer to the objective functional $F_{p,C}$ as (unbalanced) \emph{$(p,C)$-Fr\'echet functional}. 
Notably, $(p,C)$-barycenters' support is not restricted to the finite space $\X$ which raises fundamental questions on its structural properties. In the following, we establish that there exists a finite set containing the support of any $(p,C)$-barycenter (\Cref{subsec:bary}). Indeed, this set can be explicitly constructed from the support of the individual $\mu_i$'s, but its size grows exponentially in the number of individual measures. However, we prove that there always exists a \emph{sparse} $(p,C)$-barycenter whose support size is at most linear in the number of measures (\Cref{thm:bary_prop}). We note that these properties are analogs of well-known properties of Wasserstein barycenters \citep{anderes2016discrete}, that we re-establish for the unbalanced setting. \\
Comparably, employing more general entropy transport distances, we are not aware of any similar structural description of their barycenters in terms of the input measures and the parameter. Notably, the entropy optimal transport barycenter of dirac measures is not necessarily finitely supported itself \citep[for an example see][]{friesecke2021barycenters}. In contrast, our explicit structural description of the support of KR barycenters provides an immediate understanding of its properties for a given choice of $C$. This clear link between $C$ and the $(p,C)$-barycenter also allows to incorporate previous knowledge of the measures or the ground space into the choice $C$. The $(p,C)$-barycenter can be tuned to be more flexible and provide superior performance compared to its $p$-Wasserstein counterpart by avoiding to normalise each measure. An illustrative example is included in \Cref{fig:cluster_ellipses}, where the $(p,C)$-barycenter detects all clusters correctly, while the Wasserstein barycenter does not provide any structural information on the underlying measures. This showcases potentially superior robustness and flexibility of the $(p,C)$-barycenter compared to the Wasserstein barycenter. We study this comparison in more detail on multiple synthetic data sets in \Cref{sec:ComputationsSimulations}. Here, the computational results\footnote{An implementation can be found in the R package \emph{WSGeometry} on CRAN.} are based on the fact that, due to our structural analysis of the support of the $(p,C)$-barycenter, it is straightforward to modify any given state-of-the-art solver for the Wasserstein barycenter problem to solve the $(p,C)$-barycenter problem (\Cref{subsec:algos}). \\

\section{Kantorovich-Rubinstein Distance and $\mathbf{(p,C)}$-Barycenter}\label{sec:Theory}
In this section, we provide some theoretical analysis of the structural properties inherent in the UOT in \eqref{eq:UOT} and as a consequence to the KRD in \eqref{eq:KRdistance}. We also focus on the variational formulation defining the $(p,C)$-barycenter in \eqref{eq:KRbarycenter}. 

\subsection{KR Distance}

In this subsection, we focus on structural properties of minimizers for UOT in \eqref{eq:UOT} and their consequences for the KRD. Notably, one can equivalently restate the penalization of total mass in (\ref{eq:UOT}) as 
\begin{equation} \label{eq:alternativepenalization}
\begin{split}
        &C\left(\frac{\mathbb{M}(\mu)+\mathbb{M}(\nu)}{2}-\mathbb{M}(\pi)\right)\\=&\frac{C}{2}\left(\sum\limits_{x\in \X}\left(\mu(x)-\pi(x,\X)\right)+\sum\limits_{x^\prime\in \X}\left(\nu(x^\prime)-\pi(\X,x^\prime)\right)\right).
\end{split}
\end{equation}
While in \eqref{eq:UOT} the parameter $C>0$ controls the deviation of the total mass of $\pi$, the alternative representation \eqref{eq:alternativepenalization} demonstrates its marginal characterization. Indeed, the parameter $C$ specifies the maximal distance (scale) for which transportation is cheaper than creation or destruction of mass. More precisely, each optimal solution $\pi_C$ for \eqref{eq:UOT} induces a directed transportation graph $G(\pi_C)$ between the support points of $\mu$ (source points) and the support points of $\nu$ (sink points). By definition, the graph $G(\pi_C)$ contains a directed edge $(x,x^\prime)$ if and only if $\pi_C(x,x^\prime)>0$. For a directed path $P=(x_{i_1},\ldots,x_{i_k})$ in $G(\pi_C)$ its path length is defined as $\mathcal{L}(P)=\sum_{j=1}^{k-1} d^p(x_{i_j},x_{i_{j-1}})$. The parameter $C>0$ determines the maximal path length for any path in $G(\pi_C)$ as the following statement demonstrates.

\begin{lemma}\label{lem:transportC}
    For $p\geq 1$, parameter $C>0$ and measures $\mu,\nu\in\msrX$ consider the UOT \eqref{eq:UOT} with an optimal solution $\pi_C$. The length of any directed path $P$ from the corresponding transport graph $G(\pi_C)$ is bounded by
    \begin{align*}
    \mathcal{L}(P)\leq C^p.
    \end{align*}
    In particular, if $d(x,x^\prime)> C$ then for any optimal solution of \ref{eq:UOT} it holds $\pi_C(x,x^\prime)=0$.
\end{lemma}

A proof is included in \Cref{sec:proofs}. \Cref{lem:transportC} shows that the underlying transportation graph has maximal path length $C^p$ which limits the interaction between source and sink points. It will be of crucial importance for closed formulas on ultra-metric trees in the following subsection. As an immediate consequence we obtain some important statements on the KRD in \eqref{eq:KRdistance} along with its metric property.

\begin{theorem}\label{prop:KRdistance}
    For any $p\geq 1$ and parameter $C>0$ the following statements hold:
    \begin{itemize}
        \item [(i)] The $p$-th order KRD in \eqref{eq:KRdistance} defines a metric on the space of non-negative measures $\msrX$.
        \item[(ii)] If $C\leq \min_{x\neq x^\prime} d(x,x^\prime)$, then it holds that
        \begin{align*}
        \text{KR}_{p,C}^p(\mu,\nu)=\frac{C^p}{2}\text{TV}(\mu,\nu),
        \end{align*}
        where $TV(\mu,\nu)\coloneqq\nicefrac{1}{2}\sum_{x\in \X}\lvert \mu(x)-\nu(x) \rvert $ is the total variation distance. The same equality holds for all $C>0$ if $\mu(x)\geq \nu(x)$ for all $x\in \X$ or if $\mu(x)\leq \nu(x)$ for all $x\in \X$.
        \item[(iii)] If $C\geq\max_{x,x^\prime}d(x,x^\prime)$ and $\mathbb{M}(\mu)=\mathbb{M}(\nu)$, then it holds that
        \begin{align*}
        \text{KR}_{p,C}^p(\mu,\nu)=W_p^p(\mu,\nu).
        \end{align*}
        \item[(iv)] If $C_1\leq C_2$, then it holds 
        \begin{align*}
            \text{KR}_{p,C_1}^p(\mu,\nu)\leq \text{KR}_{p,C_2}^p(\mu,\nu).
        \end{align*}
    \end{itemize}
\end{theorem}

We stress that the metric property of the KRD in \Cref{prop:KRdistance} (i) has already been established in specific instances, e.g., for $p=1$ \citep{piccoli2014generalized}. Our proof follows that of Theorem $2$ in \cite{muller2020metrics} for uniform measures on point patterns with minor modifications. \\
\Cref{prop:KRdistance} demonstrates how two measures $\mu,\nu\in\msrX$ are compared with respect to KRD. Depending on the parameter $C>0$ the optimal value interpolates between $p$-th order Wasserstein distance on small scales and total variation on larger scales with respect to $C$. 
Equivalently, these properties can be shown by considerations of the \emph{dual} program for UOT in (\ref{eq:UOT}) given by
\begin{align}\label{eq:dualUOT}
\text{UOT}_{p,C}(\mu,\nu)=\max_{\substack{f,g\colon \X\to \mathbb{R}\\ f\leq \nicefrac{C^p}{2},\, g \leq \nicefrac{C^p}{2}}}\, \sum_{x\in \X} f(x)&\mu(x) + \sum_{x^\prime \in \X} g(x^\prime)\nu(x^\prime)\tag{$\text{DUOT}_{p,C}$}  \\[1ex] 
\textrm{s.t. } f(x)+g(x^\prime)&\leq d^p(x,x^\prime),\, \forall x,x^\prime \in \X,\notag
\end{align}
where the equality holds due to \emph{strong duality}. For $p=1$ this can be further specified to 
\begin{align*}
\text{UOT}_{1,C}(\mu,\nu)=\max_{\substack{f\colon \X\to \mathbb{R}\\ f\, 1-\text{Lipschitz}\\
\Vert f \Vert_{\infty}\leq \nicefrac{C}{2}}} \sum_{x\in \X} f(x)(\mu(x)&-\nu(x))
\end{align*}
which reveals its relation to the \emph{flat metric} \citep{bogachev2007measure} as observed in \cite{lellmann2014imaging,schmitzer2019framework}. As in general $\mathbb{M}(\mu)\neq \mathbb{M}(\nu)$, the bound $f,g\leq C^p/2$ on dual feasible solutions $f,g$ is necessary for the dual to be finite. However, if the measures $\mu,\nu\in\msrX$ have equal total mass $\mathbb{M}(\mu)=\mathbb{M}(\nu)$ and $C\geq \max_{x,x^\prime} d(x,x^\prime)$, then the bound on dual feasible solutions is redundant and we obtain the dual of the usual OT problem
\begin{align}\label{eq:dualOT}
\text{OT}_{p}(\mu,\nu)=\max_{f,g\colon \X\to \mathbb{R}} \sum_{x\in \X} f(x)&\mu(x) +\sum_{x^\prime \in \X}g(x^\prime)\nu(x^\prime)\notag \tag{$\text{DOT}_{p}$}\\[1ex]
\textrm{s.t. }  f(x)+ g(x^\prime) &\leq d^p(x,x^\prime),\, \forall x,x^\prime \in \X.\notag
\end{align}
\subsubsection{KR Distance on Ultrametric Trees} \label{subsec:trees}
For OT, the approximations of the underlying distance by a tree metric are common tools for theoretical and practical purposes. The former is usually employed for rates of convergence for the expectation of empirical OT costs \citep{sommerfeld2019optimal} while in the latter tree approximations serve to reduce the computational complexity inherent in OT \citep{le2019tree}. \textcolor{mygreen}{OT on ultramatric trees is also applied for the analysis of phylogenetic trees \citep{gavryushkin2016space}.} For an efficient computational implementation of UOT on tree metrics we refer to \cite{sato2020fast}. Notably, while OT with tree metric costs has a closed form solution, this fails to hold for its UOT counterpart. An exception is given in terms of ultrametric trees for which not only OT \citep{kloeckner2015geometric} but also UOT admits a closed form solution, which we establish in this subsection. \\
To this end, consider a tree $\mathcal{T}$ with nodes $V$, edges $E$ attached with (non-negative) weights $w(e)$ for $e\in E$ and a designated root \textsf{r}. Two nodes $\textsf{v},\textsf{w}\in V$ are connected by a unique path denoted $\mathcal{P}(\textsf{v},\textsf{w})$ either represented by a sequence of nodes or as a sequence of edges. The distance $d_\mathcal{T}(\textsf{v},\textsf{w})$ is equal to the sum of the weights of those edges contained in $\mathcal{P}(\textsf{v},\textsf{w})$. A \emph{leaf} of $\mathcal{T}$ is any node such that its degree (number of edges attached to the node) is equal to one and the set of all leaf nodes is denoted as $L\subset V$. A node $\textsf{v}^\star$ is termed \emph{parent} of node $\textsf{v}$ denoted by $\text{par}(\textsf{v})=\textsf{v}^\star$ if both are connected by a single edge but $\textsf{v}^\star$ is closer to the root than $\textsf{v}$. The parent of the root node is set to $\text{par}(\textsf{r})=\textsf{r}$. For a node $\textsf{v}$ its \emph{children} are the elements of the set $\mathcal{C}(\textsf{v})=\left\lbrace \textsf{w}\in V \,\mid\, \textsf{v}\in \mathcal{P}(\textsf{w},\textsf{r})\right\rbrace$. Notice that with this definition $\textsf{v}$ is a child of itself (\Cref{fig:exampletree} (a) for an illustration). 
\begin{figure}
  \centering
  \subfloat[][]{\includegraphics[width=0.47\linewidth]{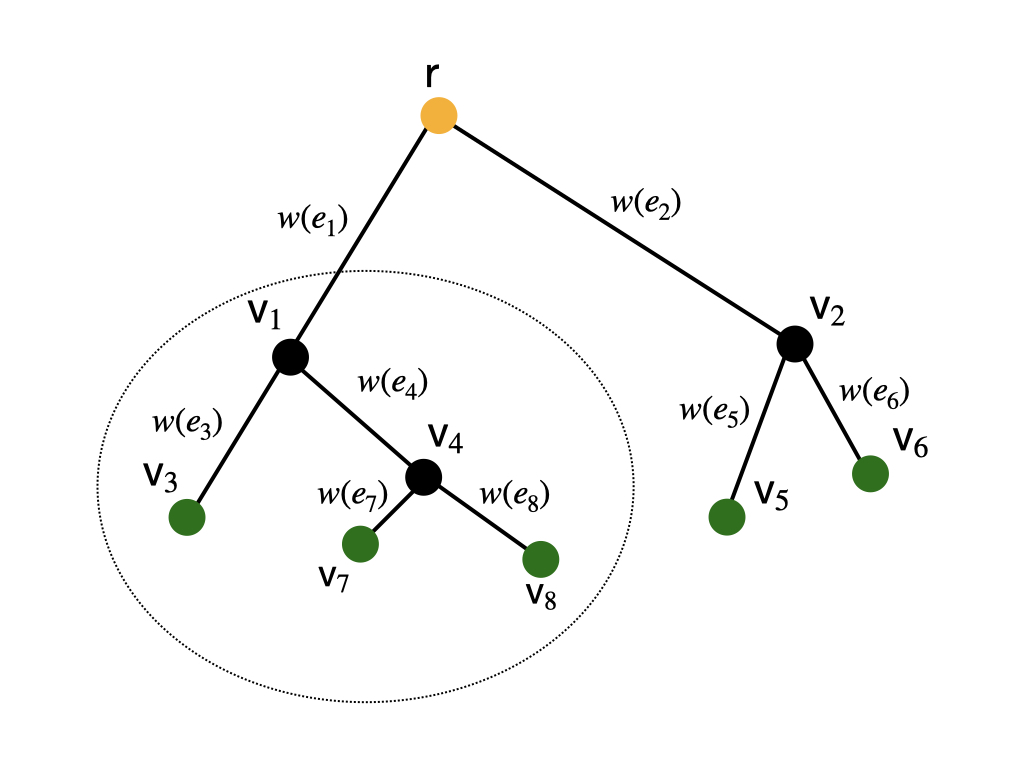}}%
  \qquad
  \subfloat[][]{\includegraphics[width=0.47\linewidth]{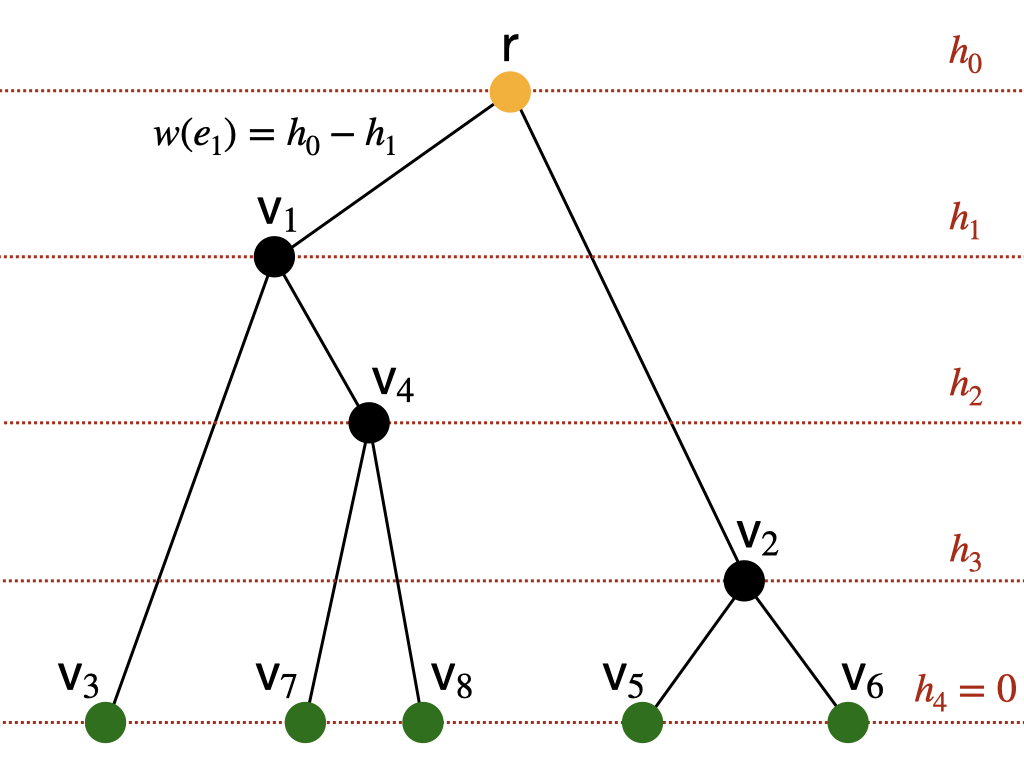}}%
  \caption{\textbf{General Tree Structures:} \textbf{(a)} A tree graph $\mathcal{T}$ with root \textsf{r} (orange), internal nodes (black) and leaf nodes $L$ (green). By definition $\text{par}(\textsf{v}_5)=\text{par}(\textsf{v}_{6})=\textsf{v}_2$ and the children of $\textsf{v}_1$ are equal $\mathcal{C}(\textsf{v}_1)=\{\textsf{v}_1,\textsf{v}_3,\textsf{v}_4,\textsf{v}_7,\textsf{v}_8\}$. The distance from each leaf node to the root may vary. \textbf{(b)} An ultrametric tree $\mathcal{T}$ with height function $h$ (red) such that $0=h_43<h_3<h_2<h_1<h_0$. Edge weights are defined by the the difference of consecutive height values, e.g. $w(e_1)=h_0-h_1$. Each leaf node (green) is at the same distance to the root \textsf{r} (orange).}%
  \label{fig:exampletree}
\end{figure}

A tree $\mathcal{T}$ is termed \emph{ultrametric tree} if all its leaf nodes are at the same distance to the root. Equivalently, there exists a \emph{height function} $h\colon V \to \mathbb{R}_+$ that is monotonically decreasing meaning that $h(\text{par}(\textsf{v}))\geq h(\textsf{v})$ and such that $h(\textsf{v})=0$ for $\textsf{v}\in L$. The distance is set to $d_\mathcal{T}(\textsf{v},\text{par}(\textsf{v}))=\left\vert h(\textsf{v})-h(\text{par}(\textsf{v}))\right\vert$ and extended on the full tree (\Cref{fig:exampletree} (b) for an illustration).\\
Consider an ultrametric tree $\mathcal{T}$ with height function $h$ and measures $\mu^L,\nu^L$ supported on the leaf nodes $L\subset V$. We prove that the $p$-th order KRD admits a \emph{closed formula} for such a setting. Intuitively, the parameter $C$ restricts transportation of mass up to a certain threshold allowing to decompose $\mathcal{T}$ into subtrees. Mass transportation is restricted solely within each subtree whereas mass abundance or deficiency is penalized with parameter $C$ for each particular subtree (\Cref{fig:explicitformula} for an illustration). We define the set
\begin{align}\label{eq:subtreeroot}
\mathcal{R}(C)\coloneqq \left\lbrace \textsf{v}\in V\, \mid \, h(\textsf{v})\leq \frac{C}{2}< h(\text{par}(\textsf{v}))\right\rbrace
\end{align} 
with the convention that $\mathcal{R}(C)=\lbrace \textsf{r} \rbrace$ if $\nicefrac{C}{2}\geq h(\textsf{r})$ and for a node $\textsf{v}\in V$ set
\begin{align*}
\mu^L(\mathcal{C}(\textsf{v}))\coloneqq \sum_{\textsf{w}\in\mathcal{C}(v)\textcolor{purple}{\cap L}} \mu^L(w).
\end{align*}
\begin{figure}
  \centering
  \subfloat[][]{\includegraphics[width=0.47\linewidth]{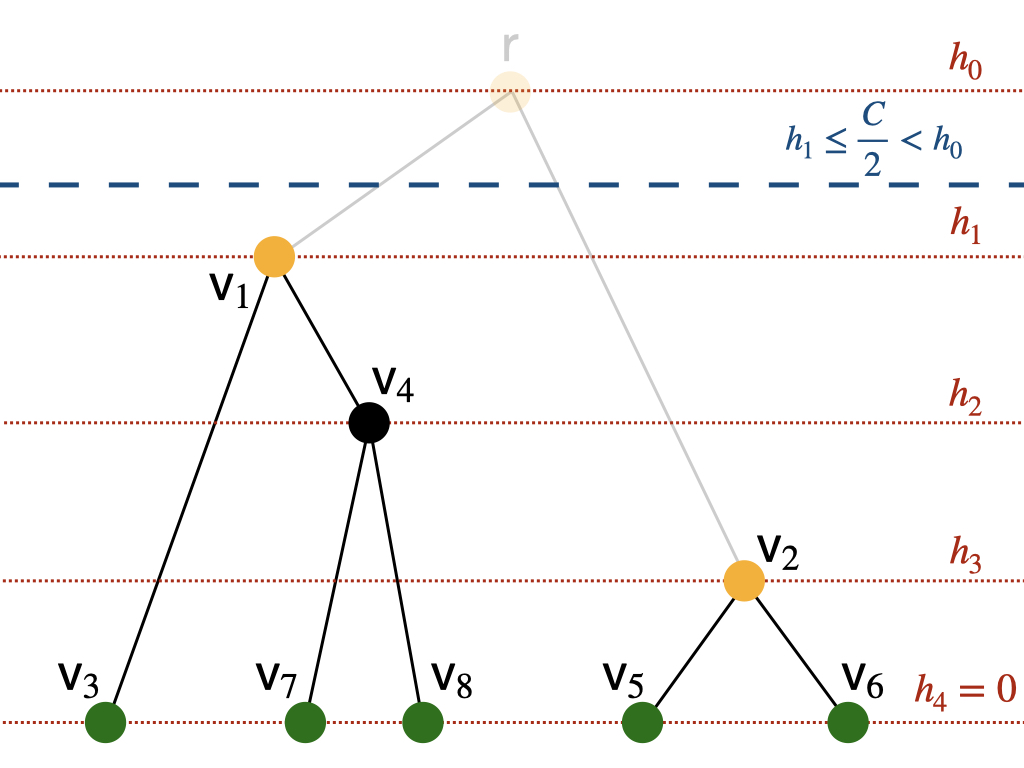}}%
  \qquad
  \subfloat[][]{\includegraphics[width=0.47\linewidth]{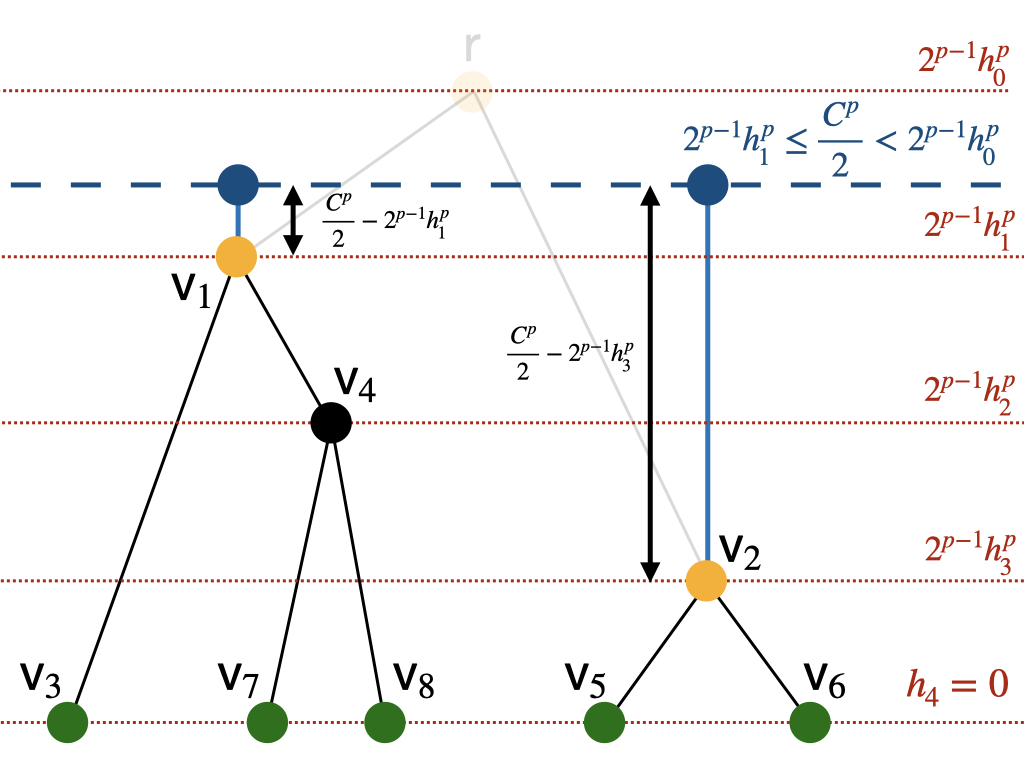}}%
  \caption{\textbf{Closed formula for the KRD on ultrametric trees:} \textbf{(a)} Depending on the regularization $C>0$ and the underlying height function $h$ the ultrametric tree $\mathcal{T}$ introduced in \Cref{fig:exampletree} (b) is decomposed into two subtrees. Each node in the set $\mathcal{R}(C)=\left\lbrace \textsf{v}_1,\textsf{v}_2\right\rbrace$ (orange) serves as a new root and corresponding subtrees $\mathcal{T}(\textsf{v}_1)\coloneqq \mathcal{C}(\textsf{v}_1)$ and $\mathcal{T}(\textsf{v}_2)\coloneqq \mathcal{C}(\textsf{v}_2)$ are equal their respective set of children with corresponding edges. \textbf{(b)} The $p$-th height transformation $\mathcal{T}_p(\textsf{v}_1)$ and $\mathcal{T}_p(\textsf{v}_2)$ of the induced subtrees $\mathcal{T}(\textsf{v}_1)$ and $\mathcal{T}(\textsf{v}_2)$, respectively. Each subtree is extended by a new root (blue) with an edge (lightblue) whose distance is equal the difference of regularization $\nicefrac{C^p}{2}$ and the $p$-th height transformed value of the former root.}%
  \label{fig:explicitformula}
\end{figure}

\begin{theorem}[KR on ultrametric trees]\label{thm:KRultrametric}
Consider an ultrametric tree $\mathcal{T}$ with leaf nodes $L$ and height function $h\colon V\to \mathbb{R}_+$ inducing the tree metric $d_\mathcal{T}$. For any $p\geq 1$ and two measures $\mu^L,\nu^L\in \mathcal{M}_+(L)$ supported on the leaf nodes of $\mathcal{T}$ it holds that
{\begin{align*}
&\text{KR}^p_{p,C}\left(\mu^L,\nu^L\right)=\\
 &\sum_{\textsf{v}\in \mathcal{R}(C)} \Bigg(2^{p-1}\sum_{\textsf{w}\in \mathcal{C}(\textsf{v})\setminus \lbrace \textsf{v} \rbrace}  \Big(\left( h(par(\textsf{w}))^p-h(\textsf{w})^p \right) \left\vert \mu^L(\mathcal{C}(\textsf{w}))-\nu^L(\mathcal{C}(\textsf{w}))\right\vert \Big) \\[1ex]
&+\left(\frac{C^p}{2}-2^{p-1}h(\textsf{v})^p\right) \left\vert \mu^L(\mathcal{C}(\textsf{v}))-\nu^L(\mathcal{C}(\textsf{v}))\right\vert \Bigg).
\end{align*}}
\end{theorem} 
The closed formula in \Cref{thm:KRultrametric} decomposes the underlying UOT into two tasks. While summing over subtrees carried out by the outer sum, the inner sum consists of two terms. The first considers OT within each subtree whereas the second accounts for mass deviation on that particular subtree.

The proof of this formula is given in \Cref{app:tree}.
\subsection{$\mathbf{(p,C)}$-Barycenters}\label{subsec:bary}
\begin{figure}
    \centering
    \includegraphics[width=\textwidth]{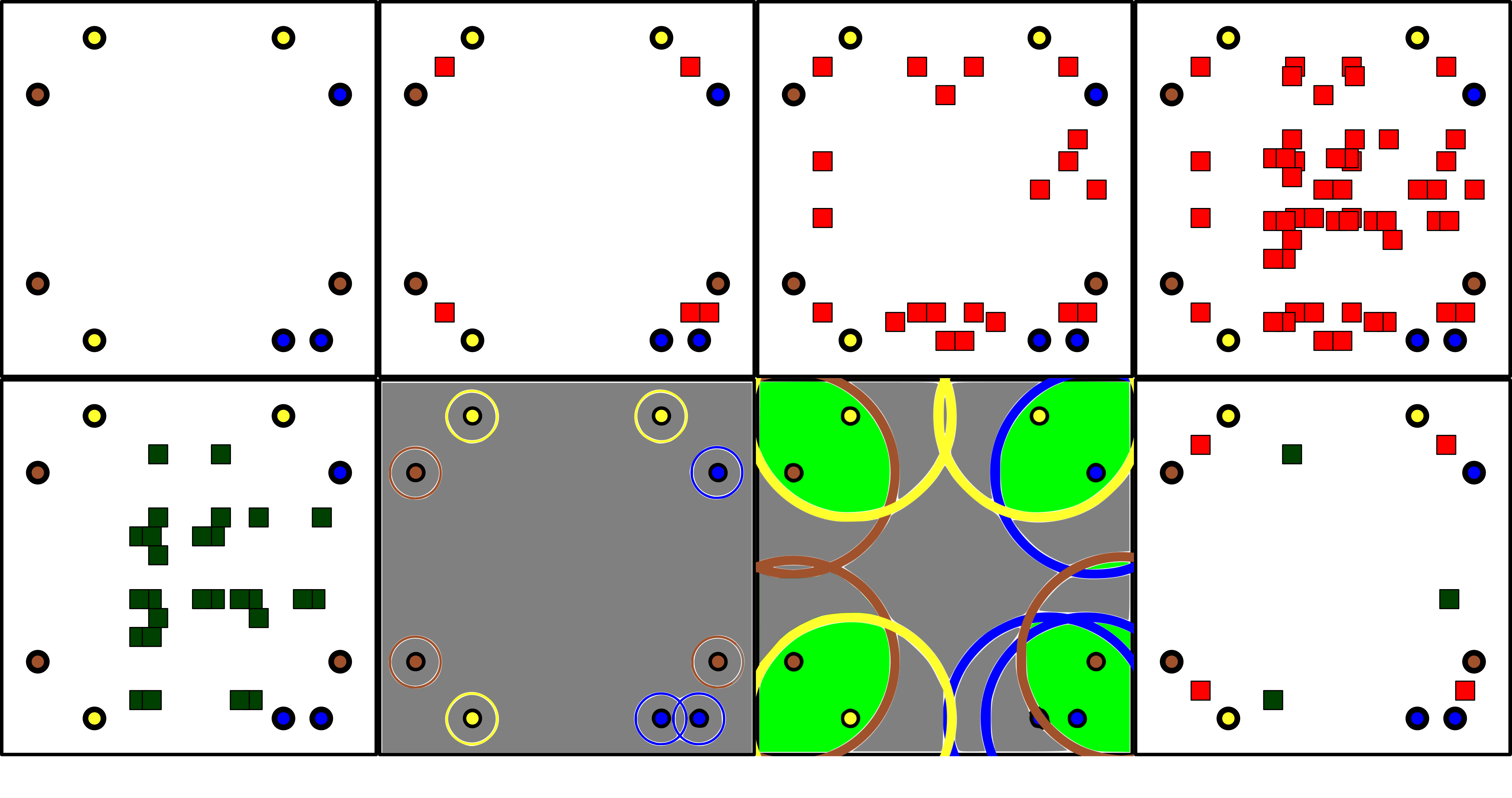}
    \caption{Centroid sets and barycenters: The support points of three ($J=3)$ measures (yellow, brown and blue dots) with unit mass at each position. \textbf{Top:} Different centroid sets $\mathcal{C}_{KR}(3,2,C)$ (red squares) with increasing value of $C$ from left to right. \textbf{Bottom-Left:} The centroid set $\mathcal{C}^{W}(3,2)$ (dark green squares) corresponding to the $2$-Wasserstein barycenter. \textbf{Bottom-Center:} Circles corresponding to $C^p$ (for two different choices of $C$) balls around the support points. The grey colouring indicates that there is no overlap of at least two circles in this area and thus no $(2,C)$-barycenter can have mass in this area. Conversely, the green colouring indicates overlap and thus the potential support area of the barycenter. \textbf{Bottom-Right:} The $(2,C)$-barycenter (red squares) for a specific choice of $C$ and the $2$-Wasserstein barycenter (dark green squares).}
    \label{fig:centroidandbarycenter}
\end{figure}
In the finite setting considered in this work a $(p,C)$-barycenter as defined in \eqref{eq:KRbarycenter} always exists, but is not necessarily unique. Moreover, the location and structure of the support of the $(p,C)$-barycenter are not fixed and hence unknown. For the Wasserstein barycenter there exists a finitely supported, sparse barycenter in this context \citep{anderes2016discrete,le2017existence}. We establish analog properties of the $(p,C)$-barycenter. 
\begin{definition}
Let $(\Y ,d)$ be a metric space, $p\geq 1$ and $J\in\mathbb{N}$. A Borel barycenter application $T^{J,p}$ associates to any points $(y_1,\dots,y_J)\in \mathcal{Y}^J$ a minimum $y^\star\in\Y$ of $\sum_{i=1}^J d^p(y_i,y)$, i.e.,
\begin{align*}
T^{J,p}(y_1,\dots ,y_J) \in  \argmin_{y \in \Y} \sum_{i=1}^J d^p(y_i,y).
\end{align*}
\end{definition}
A Borel barycenter application is in general not a function since the minimum does not need to be unique. In particular, $y=T^{J,p}(y_1,\dots,y_J)$ only means that $y$ is one of the minima of the average distance function. As the measures $\mu_1,\ldots,\mu_J$ are defined on $\X$ we usually restrict the Borel barycenter application to inputs from the space $\X\subset \Y$. We define the \emph{full centroid set} of the measures $\mu_1,\ldots,\mu_J\in\msrX$ as
\textcolor{mygreen}{
\begin{align}\label{eq:KRfullcentroid}
\begin{split}
\mathcal{C}_{KR}(J,p)=\Big\{ y \in \mathcal{Y} \ \vert \ \exists L\geq \lceil J/2 \rceil, \ \exists (i_1,\dots,i_L)\subset \{1,\dots,J\},\\ x_1,\dots,x_L: \ x_l\in \text{supp}(\mu_{i_l}) \\ \forall l=1,\dots,L: \ 
y=T^{L,p}(x_1,\dots,x_L)\Big\},
 \end{split}
\end{align}
and the \emph{restricted centroid set} 
\begin{align}
\begin{split}
\label{eq:KRCentroid}
\mathcal{C}_{KR}(J,p,C)=\Big\{ &y=T^{L,p}(x_1,\dots,x_L)\in\mathcal{C}_{KR}(J,p)\, \mid \forall 1\leq l \leq L: \\ &d^p(x_l,y)\leq C^p; \ \es{l=1}{L}d^p(x_l,y) \leq \frac{C^p (2L-J)}{2} \Big\}. 
\end{split}
\end{align}
}
\textcolor{mygreen}{We stress that for each $L$-tupel $(x_1,\dots,x_L)$ one fixed representative of $T^{L,p}(x_1,\dots,x_L)$ is chosen for the construction of the centroid set $C_{KR}(J,p,C)$. To streamline the presentation any statement concerning $\mathcal{C}_{KR}(J,p,C)$ in the following theorem is to be understood in the sense that there exists a choice of $\mathcal{C}_{KR}(J,p,C)$ such that the statement holds true.}
\begin{theorem}\label{thm:bary_prop}
Let $\mu_1,\dots ,\mu_J \in \msrX$ be a collection of non-negative measures on the finite discrete space $\X\subset \Y$. For any $C>0$ it holds that
\begin{itemize}
\item[(i)] $$\inf_{\mu \in \msrY}F_{p,C}(\mu)=\inf_{\substack{\mu \in \mathcal{M}_+(\Y)\\\text{supp}(\mu)\subseteq\mathcal{C}_{KR}(J,p,C) }}F_{p,C}(\mu).$$ Moreover, any $(p,C)$-barycenter $\mu^\star$ satisfies $\text{supp}(\mu^\star)\subseteq \mathcal{C}_{KR}(J,p,C)$ and its total mass is bounded by \[0\leq\mathbb{M}(\mu^\star)\leq \frac{2}{J}\sum_{i=1}^J \mathbb{M}(\mu_i).\]
\item[(ii)] \textcolor{mygreen}{For any $(p,C)$-barycenter $\mu^*$ and any point $y \in \text{supp}(\mu^*)$, there exist UOT plans $\pi_i$ between $\mu^*$ and $\mu_i$ for $i=1,\dots,J$, respectively, such that if $\pi_i(y,x)>0$, then there exists $L\geq \lceil J/2 \rceil$, $x_l\in \text{supp}(\mu_{i_l})$ for $l=2,\dots,L$, $(i_2,\dots,i_L)\subset \{1,\dots ,J\}$ and $i_l\neq i$ for $l=2,\dots ,L$ with $y=T^{L,p}(x,x_{i_2},\dots,x_{i_L})$, $\pi_{j}(y,x_{j})>0$ if $j \in \{i_2,\dots,i_L \}$.\\
Additionally, if for any $(x_1,\dots,x_L)\in \Y^L$ it holds that
\begin{align}\label{eq:inj}
    T^{L,p}(x_1,\dots,x_L)=T^{L,p}(y_1,x_2,\dots,x_L) \Leftrightarrow x_1=y_1,
\end{align}
then $\pi_i(y,x)\in \{0,\mu^*(y)\}$ for $i=1,\dots,J$.}
\item[(iii)] If $M_i\coloneqq\vert\text{supp}(\mu_i)\vert$ for $1\leq i\leq J$ then there exists a $(p,C)$-barycenter $\mu^\star$ such that $$\lvert \text{supp}(\mu^\star)  \rvert \leq \min \left\{ \lvert \mathcal{C}_{KR}(J,p,C) \rvert , \sum_{i=1}^J M_i \right\}.$$
\item[(iv)] If $C_1\leq C_2$, then it holds
\[
\inf_{\mu \in \msrY}F_{p,C_1}(\mu)\leq \inf_{\mu \in \msrY}F_{p,C_2}(\mu).
\]
\item[(v)] Furthermore, set $\mathcal{Z}\coloneqq \bigcup_{i=1}^J \text{supp}(\mu_i)\cup\mathcal{C}_{KR}(J,p)$ and define $$d_{\min}^\prime\coloneqq\min_{x\in \mathcal{Z}\setminus\mathcal{C}_{KR}(J,p), \ y\in \mathcal{C}_{KR}(J,p)} d(x,y).$$ 
If $C\leq d_{\min}^\prime$, then the $(p,C)$-barycenter $\mu^\star$ is given by
$$
\mu^\star=\es{x \in \X}{} med(\mu_1(x),\dots ,\mu_J(x))\delta_{x}.
$$
\item[(vi)] Let $C>J^{1/p}\text{diam}(\mathcal{Z})$ and let $\mu_1,\dots ,\mu_J$ be ordered such that $\mathbb{M}(\mu_i)\leq \mathbb{M}(\mu_j)$ for $i\leq j$. Suppose that $J$ is odd or there there exists no point $y \in \Y$ contained in at least $J/2$ different support sets. Then, for any $(p,C)$-barycenter $\mu^\star$ it holds that $\mathbb{M}(\mu^\star)= \mathbb{M}\left(\mu_{\lceil J/2 \rceil}\right)$. Else, there exists at least one $(p,C)$-barycenter with this total mass.
\end{itemize}
\end{theorem}
The proof is based on the fact that finding a $(p,C)$-barycenter can be proven to be equivalent to solving a \emph{multi-marginal optimal transport problem} (\Cref{app:liftbary}). Statement $(i)$ provides insights into the structure of the support of any $(p,C)$-barycenter and its dependency with respect to the magnitude of $C$. The definition of $\mathcal{C}_{KR}(J,p,C)$ can be understood as a joint restriction on $\sum_{i=1}^{L}d^p(x_i,y)$ combined with an individual restriction on each $d^p(x_i,y)$  of the original centroid points of $\mathcal{C}_{KR}(J,p)$. The joint restriction ensures that simply deleting any mass at a given centroid point (and thus reducing the total mass of the measure) does not improve the objective value. This is a minimal feasibility assumption on the considered centroid point, as otherwise no measure containing this point can be optimal. The second restriction concerns each point individually. If a point $x_i$ has a distance larger than $C^p$ from a point $y$, then, by \Cref{lem:transportC}, there is no transport between $y$ and $x_i$. Thus, centroids which have have a larger distance to one of the points $x_1,\dots,x_L$ they are constructed from can not be in the support of any $(p,C)$-barycenter. This also gives rise to some helpful intuition for the support structure of any $(p,C)$-barycenter. Considering all $C^p$-neighbourhoods around any of the support points of the $\mu_i$, then a $(p,C)$-barycenter can only have support in regions where at least balls from $\lceil J/2 \rceil$ different measures intersect. A visual representation of this is given in the center of bottom row of \Cref{fig:centroidandbarycenter}. By definition, the sets $\mathcal{C}_{KR}(J,p,C)$ are equipped with a natural ordering in the sense that if $C_1\leq C_2$ then $\mathcal{C}_{KR}(J,p,C_1)\subseteq \mathcal{C}_{KR}(J,p,C_2)$. Moreover, if $C$ is large enough then $\mathcal{C}_{KR}(J,p,C)=\mathcal{C}_{KR}(J,p)$. We illustrate these sets in the top row \Cref{fig:centroidandbarycenter}. We observe that the cardinality of the restricted centroid set in \eqref{eq:KRCentroid} decreases with decreasing $C$. In the extremes for large $C$ the restricted centroid sets coincides with the full centroid sets in \eqref{eq:KRfullcentroid} that is independent of $C$. For small $C$, if there is no point which is contained in the support of at least $J/2$ measures, the restricted centroid set is empty. For an illustration we refer to the top row of \Cref{fig:centroidandbarycenter}.\\
\textcolor{mygreen}{Property $(ii)$ is an analogue to a well-known characterization \citep{anderes2016discrete} of the $p$-Wasserstein barycenter on $\mathbb{R}^d$ with Euclidean distance $d_2$, where the transport from the barycenter to the underlying measures is characterized by a transport map.  The corresponding statement for the $(p,C)$-barycenter holds true as well in this context. Indeed, on $(\mathbb{R}^d,d_2)$ condition \eqref{eq:inj}, which can be understood as an injectivity-type assumption on the barycentric application, is satisfied due to the fact that $T^L(x_1,\dots,x_L)=\frac{1}{L}\sum_{l=1}^Lx_l$. However, for $(\mathbb{R}^d,d_1)$ this assertion does not hold. Consider $x_1<x_2<x_3<x_4\in \mathbb{R}$ and measures $\mu_1=\delta_{x_1}+\delta_{x_2}, \mu_2=\delta_{x_3}+\delta_{x_4}$, then any measure of the form $\mu^*=2\delta_{y}$ for any $y \in [x_2,x_3]$ is a $(p,C)$-barycenter for $C>2\lvert x_1-x_4\rvert$. Thus, there only exist mass-splitting UOT plans between $\mu^*$ and $\mu_1,\mu_2$ and the transport is not characterized by a transport map. On more general spaces such as a tree $\mathcal{T}$ rooted at $r$, three leaves $x_1,x_2,x_3$ and positive edge weights $e_1,\dots, e_3\in (0,1)$ the barycenter on $\mathcal{T}$ of any two leafs $x_i\neq x_j$, is the root $r$. In particular, in this example, or in fact in any tree $\mathcal{T}=(V,E)$ which has a vertex $y$ with degree of at least three\footnote{The degree of a vertex in a graph is the number of vertices which are adjacent to it.} condition \eqref{eq:inj} fails. The unique $(2,2)$-barycenter of two measures $\mu_1=\delta_{x_1}+\delta_{x_2}$ and $\mu_2=\delta_{x_2}+\delta_{x_3}$ is given by $\mu^*=2\delta_r$. Thus, there are again only mass-splitting UOT plans between $\mu^*$ and $\mu_1$ and $\mu_2$. However, for the unit circle $\mathcal{S}^1$ equipped with its natural arc-length distance property \eqref{eq:inj} does hold. Assume $a_0=T^L(x_1,\dots,x_L)=T^{L,p}(y_1,\dots,x_L)$, $a_1=T^{L-1,p}(x_2,\dots,x_L)$ and for each $x\in \mathcal{S}_1$ denote $H_r(x)$ and $H_l(x)$ as the halfcircle right and left of $x$, respectively. It is straightforward to see by contraposition that if it holds $a_1\in H_r(a_0)$, then this implies $x_1,y_1\in H_l(a_1)$ and $x_1,y_1\in H_l(a_0)$. However, it also holds $d(x_1,a_0)=d(y_1,a_0)$, and thus $\langle x_1-y_1,a_0 \rangle=0$. In particular, this implies that either $x_1\in H_l(a_0)$ and $y_1\in H_r(a_0)$ or vice versa and hence $x_1=y_1$. The case $a_1\in H_l(a_0)$ is analog and the case $a_0=a_1$ clear. 
}\\
Property $(iii)$ guarantees the existence of \textit{sparse} $(p,C)$-barycenters. For large $C$ the size $\mathcal{C}_{KR}(J,p,C)$ scales as $\prod_{i=1}^{J} M_i$, growing essentially exponentially in $J$. However, here we see that there always exists a $(p,C)$-barycenter supported on a sparse subset of $\mathcal{C}_{KR}(J,p,C)$ which has cardinality growing only linearly in $J$. Part $(iv)$ simply extends the montonicity of the $(p,C)$-KRD to the $(p,C)$-Fr\'echet functional. Statement $(v)$ yields a critical point after which decreasing $C$ does no longer change the resulting $(p,C)$-barycenter and provides a closed form characterisation of the $(p,C)$-barycenter in this context. Finally, statement $(vi)$ enables control on the total mass of the $(p,C)$-barycenter for large values of $C$. In particular, since the total mass is close to the median of the total masses of the $\mu_i$, we point out that the total mass of the $(p,C)$-barycenter in this setting is robust against outliers. A small amount of measures with unreasonably high mass has no impact on the total mass of the $(p,C)$-barycenter. \\
Naturally, we compare the $(p,C)$-barycenter to its popular Wasserstein analogue in (\ref{eq:otbarycenter}). As proven in \cite{le2017existence} \citep[and initially for $p=2$ for $\mathbb{R}^d$ by][]{anderes2016discrete} the support of any $p$-Wasserstein barycenter is contained in
\begin{align}\label{eq:WassersteinCentroid}
\mathcal{C}_{W}(J,p)=\left\{ y \in \mathcal{Y} \ \vert \ y=T^{J,p}(x_1,\dots ,x_J), x_i \in \text{supp}(\mu_i)  \right\}.
\end{align}
Compared to the $p$-Wasserstein barycenter of the probability measures $\mu_1,\ldots,\mu_J$ the restricted centroid set $\mathcal{C}_{KR}(J,p,C)$ allows more flexibility for specific cases and can provide a more reasonable representation of the data. We illustrate this in \Cref{fig:centroidandbarycenter} (bottom-left/right) where the $(2,C)$-barycenter clearly represents all clusters while the $2$-Wasserstein barycenter fails to capture them. Nevertheless, if $C$ is large enough and all measures have equal total mass both barycenters coincide.
\begin{corollary}
If $C> 2^{\frac{1}{p}}\text{diam}(\mathcal{Z})$ and $\mathbb{M}(\mu_1)=\mathbb{M}(\mu_2)=\dots =\mathbb{M}(\mu_J)$, then any $p$-Wasserstein barycenter is also a $(p,C)$-barycenter and vice versa.
\end{corollary}

While this shows that the $(p,C)$-barycenter is a strict generalisation of the usual $p$-Wasserstein barycenter as the solutions coincide for large $C$, for smaller values of $C$ there can be significant differences. One such striking difference between the $p$-Wasserstein barycenter and the $(p,C)$-barycenter comes in the form of a localization property. Let $B_1,\dots ,B_R \subset \Y$ such that $\text{supp}(\mu_i)\subset \cup_{r=1}^{R}B_r$ with $\text{diam}(B_r)\leq C$ for all $r=1,\dots ,R$ and $d(B_k,B_l)>2^{1/p} C$ for all $k\neq l$. \textcolor{mygreen}{Here, the $(p,C)-$barycenter tends to place mass between the clusters $B_1,\dots,B_R$. However, a $(p,C)$-barycenter is obtained by combining $R$ barycenters of the measures restricted to the $B_1,\dots,B_R$, respectively.}
\begin{lemma}
\label{lem:cluster}
Let $\mu_1,\dots ,\mu_J \in \msrX$ such that for all $i=1,\dots ,J$ it holds $\text{supp}(\mu_i)\subset \cup_{r=1}^{R}B_r$ for some $B_1,\dots ,B_R \subset \Y$ with $\text{diam}(B_r)\leq C$ for all $r=1,\dots ,R$ and $d(B_k,B_l)> 2^{1/p}C$ for all $k\neq l$. For $r=1,\dots ,R$, let
\[
\mu^*_r \in \argmin_{\mu\in\mathcal{M}_+(conv(B_r))} \frac{1}{J}\sum_{i=1}^J \text{KR}_{p,C}^p(\mu,{\mu_i}_{\vert B_r}),
\]
where $conv(B_r)$ is the convex hull of $B_r$ for $r=1,\dots ,R$. Then, the measure $\es{r=1}{R} \mu_r^*$ is a $(p,C)$-barycenter of $\mu_1,\dots ,\mu_J$.
\end{lemma}
\textcolor{mygreen}{In particular, \Cref{lem:cluster} implies that the $(p,C)$-barycenter respects the cluster structure within the supports of the measures if the clustered are sufficiently separated and $C$ is adapted according to the cluster size. Examples of this setting can be seen in \Cref{fig:cluster_ellipses} and \Cref{fig:centroidandbarycenter}.} 
\section{A Lift to Optimal Transport, Wasserstein Barycenters and Multi-Marginal Optimal Transport} \label{sec:pac}
In this section, we provide the necessary tools and framework to establish our results in the previous section. Following the ideas of \cite{guittet2002extended} we state UOT in \eqref{eq:UOT} as an equivalent balanced OT problem. We extend this idea to the $(p,C)$-barycenter, showing it to be equivalent to a specific Wasserstein barycenter problem as well as a balanced multi-marginal optimal transport problem. 
\subsection{A Lift to Optimal Transport}\label{app:lifttoOT}
We fix a parameter $C>0$, introduce an additional dummy point $\dum$ and define the augmented space $\tilde{\X}\coloneqq \X\cup \{\dum\}$ with metric cost
\begin{align}\label{eq:augmetric}
\tilde{d}^p_C(x,x^\prime)=
\begin{cases}
d^p(x,x^\prime)\wedge C^p,\, &x,x^\prime\in \mathcal{X},\\[1ex]
\frac{C^p}{2},\, &x\in\mathcal{X},\,x^\prime=\dum,\\[1ex]
\frac{C^p}{2},\, &x=\dum,\,x^\prime\in\mathcal{X},\\[1ex]
0,\, &x=x^\prime=\dum.
\end{cases}
\end{align}
Notably, $\tilde{d}_C\colon \tilde{\mathcal{X}}\times\tilde{\mathcal{X}}\to \mathbb{R}_+$ defines a metric on $\tilde{\mathcal{X}}$ \cite[Lemma A1]{muller2020metrics}. Consider the subset $\mathcal{M}_+^B(\X)\coloneqq \left\lbrace \mu \in \msrX\, \mid\, \mathbb{M}(\mu)\leq B\right\rbrace\subset \mathcal{M}_+(\mathcal{X})$ of non-negative measures whose total mass is bounded by $B$. Setting $\tilde{\mu}\coloneqq\mu + (B-\mathbb{M}(\mu))\delta_\dum$, any measure $\mu \in \mathcal{M}_+^B(\mathcal{X})$ defines an \emph{augmented measure} $\tilde{\mu}$ on $\tilde{\X}$ such that $\mathbb{M}(\tilde{\mu})=B$. Hence, for two measures $\mu,\nu\in\mathcal{M}_+^B(\X)$ we can define the OT problem on $\tilde{\X}$ between their augmented measures $\tilde{\text{OT}}_{\tilde{d}_C^p}(\tilde{\mu},\tilde{\nu})$. In fact, it holds that 
\begin{align*}
\text{UOT}_{p,C}(\mu,\nu)=\text{UOT}_{d^p\wedge C^p,C}(\mu,\nu)=\tilde{\text{OT}}_{\tilde{d}^p_C}(\tilde{\mu},\tilde{\nu}),
\end{align*}
where the first equality follows by \Cref{lem:transportC} as for any optimal solution $\pi_C$ it holds $\pi_C(x,x^\prime)=0$ if $d^p(x,x^\prime)> C^p$ and the second follows by \cite[Lemma 3.1]{guittet2002extended}. The same equalities remain valid replacing $B$ by an arbitrarily large constant as summarized by the following lemma.

\begin{lemma}\label{lem:add_mass}
Consider $\mu,\nu \in \mathcal{M}_+^B(\X)$ with extended versions $\tilde{\mu},\tilde{\nu}$. Then for any $a>0$ it holds that
\begin{align*}
\tilde{\text{OT}}_{\tilde{d}^p_C}(\tilde{\mu},\tilde{\nu})=\tilde{\text{OT}}_{\tilde{d}^p_C}(\tilde{\mu}+a\delta_{\dum},\tilde{\nu}+a\delta_{\dum}).
\end{align*}
\end{lemma}

\begin{proof}
For $p=1$, the result is trivial since by duality $\tilde{\text{OT}}_{\tilde{d}_C}(\tilde{\mu},\tilde{\nu})$ only depends on the difference of the measures. For $p>1$ we invoke $\tilde{d}_C$-cyclical monotonicity \cite[Thm. 5.10]{villani2008optimal} of any OT plan $\pi$ and use the property that $\tilde{d}_C^p(x,\dum)=\nicefrac{C^p}{2}$. This yields that $(\dum,\dum)\in \text{supp}(\pi)$ which leads to the desired conclusion.
\end{proof}

\subsection{A Lift to Wasserstein Barycenters}\label{app:liftbary}

We can also lift the optimization problem defining a $(p,C)$-barycenter to an equivalent $p$-Wasserstein barycenter formulation \eqref{eq:otbarycenter}. Augmentation of the underlying measures, however, is not straightforward as the total mass of the $(p,C)$-barycenter is unknown. A first crude upper bound on its total mass leads to a feasible approach.

\begin{lemma}\label{lem:mass_cap}
Consider $\mu_1,\dots ,\mu_J \in \msrX$ and let $F_{p,C}$ be their associated unbalanced Fr\'echet functional. Then it holds that
\begin{align*}
\argmin_{\mu \in \mathcal{M}_+(\Y)} F_{p,C}(\mu) = \argmin_{\substack{\mu \in \mathcal{M}_+(\Y) \\ \mathbb{M}(\mu)\leq \sum_{i=1}^J \mathbb{M}(\mu_i)}} F_{p,C}(\mu).
\end{align*}
More precisely, any $(p,C)$-barycenter $\mu^\star$ of $\mu_1,\ldots,\mu_J$ satisfies $\mathbb{M}(\mu^\star)\leq\sum_{i=1}^J \mathbb{M}(\mu_i)$.
\end{lemma}

\begin{proof}
Assume first that there exists a measure $\mu \in \mathcal{M}_+(\Y)$ such that $\mu=\nu_1+\nu_2$ where no transport between $\nu_2$ and any $\mu_i$ occurs in the optimal solution of $\text{UOT}_{p,C}(\mu,\mu_i)$ for $1\leq i \leq J$ and it holds $\mathbb{M}(\nu_2)>0$. Thus it holds
\begin{align*}
F_{p,C}(\mu)=F_{p,C}(\nu_1+\nu_2)= F_{p,C}(\nu_1)+(C^p/2) \mathbb{M}(\nu_2) > F_{p,C}(\nu_1)
\end{align*}
and we improve the objective value of $\mu$ by removing $\nu_2$. Hence, let $\mu\in\mathcal{M}_+(\Y)$ be any measure such that $\nu_2\equiv0$. Consider $\pi_i$ the optimal solution for $\text{UOT}_{p,C}(\mu,\mu_i)$ for each $1\leq i \leq J$. Decompose the measure $\mu=\sum_{i=1}^J \tau_i$, where $\tau_i$ is the mass of $\mu$ transported to $\mu_i$ according to $\pi_i$ and which is not yet included in any $\tau_j$ for $j<i$. Clearly, $\mathbb{M}(\mu)=\sum_{i=1}^J \mathbb{M}(\tau_i)\leq \sum_{i=1}^J \mathbb{M}(\mu_i)$ and we conclude that
\begin{align*}
\min_{\mu \in \mathcal{M}_+(\Y)} F_{p,C}(\mu) = \min_{\substack{\mu \in \mathcal{M}_+(\Y) \\ \mathbb{M}(\mu)\leq \sum_{i=1}^J \mathbb{M}(\mu_i)}} F_{p,C}(\mu).
\end{align*}
By our first considerations the claim follows.
\end{proof}

Given the upper bound on the total mass of any $(p,C)$-barycenter at our disposal we can formulate a lift of the $(p,C)$-barycenter problem to a related $p$-Wasserstein barycenter problem. For this, let $\tilde{\Y}\coloneqq\Y\cup\{\dum\}$ endowed with the metric $\tilde{d}_C$ in \eqref{eq:augmetric} (replace $\X$ by $\Y$ and recall that $\X\subset\Y$) and augment the measures $\mu_1,\ldots,\mu_J$ to $\tilde{\mu}_1,\ldots,\tilde{\mu}_J$ where $\tilde{\mu}_i=\mu_i+\sum_{j\neq i} \mathbb{M}(\mu_j) \delta_\dum$ for $1\leq i\leq J$. In particular, $\mathbb{M}(\tilde{\mu}_i)=\sum_{j=1}^J \mathbb{M}(\mu_j)$ and we can define the \emph{augmented $p$-Fr\'echet functional} 
\begin{align*}
\tilde{F}_{p,C}(\mu):=\frac{1}{J}\sum_{i=1}^{J}\tilde{\text{OT}}_{\tilde{d}^p_C}^p(\tilde{\mu}_i,\mu),
\end{align*}
where by definition $\tilde{F}_{p,C}$ is restricted to measures $\mu$ with mass $\mathbb{M}(\mu)=\sum_{i=1}^J \mathbb{M}(\mu_i)$.
 
\begin{lemma}\label{lem:fre_eq}
For $1\leq i \leq J$ consider measures $\mu_i\in\msrX$ and their augmented versions $\tilde{\mu}_i\coloneqq \mu_i+\sum_{j\neq i}\mathbb{M}(\mu_j)\delta_\dum$, respectively. Then it holds that
\begin{align*}
F_{p,C}(\mu)=\tilde{F}_{p,C}\left( \mu+\left( \sum_{i=1}^J \mathbb{M}(\mu_i)-\mathbb{M}(\mu)\right) \delta_{\dum} \right)
\end{align*}
for all $\mu\in\mathcal{M}_+(\Y)$ such that $\mathbb{M}(\mu)\leq \sum_{i=1}^J \mathbb{M}(\mu_i)$ and in particular 
\begin{align*}
\min_{\mu \in \mathcal{M}_+(\Y)} F_{p,C}(\mu) = \min_{\substack{\mu \in \mathcal{M}_+(\tilde{\Y}) \\ \mathbb{M}(\mu)= \sum_{i=1}^J \mathbb{M}(\mu_i)}} \tilde{F}_{p,C}(\mu).
\end{align*}
\end{lemma}
The proof of this Lemma is given in \Cref{sec:proofpac}.

\begin{remark}[Optimal $(p,C)$-barycenters]
\Cref{lem:fre_eq} states that the optimal objective value for the $(p,C)$-barycenter is equal the related $p$-Wasserstein barycenter problem on the augmented space. In particular, the proof also reveals that if $\tilde{\mu}^\star$ is a $p$-Wasserstein barycenter for the augmented measures $\tilde{\mu}_1,\ldots,\tilde{\mu}_J$ then $\mu^\star\coloneqq\tilde{\mu}^\star-\tilde{\mu}^\star(\dum)\delta_\dum$ is a $(p,C)$-barycenter for the measures $\mu_1,\ldots,\mu_J$. Vice versa, if $\mu^\star$ is a $(p,C)$-barycenter for the measures $\mu_1,\ldots,\mu_J$ then $\tilde{\mu}^\star\coloneqq \mu^\star+\left( \sum_{i=1}^J \mathbb{M}(\mu_i)-\mathbb{M}(\mu^\star)\right)\delta_\dum$ is a $p$-Wasserstein barycenter for the augmented measures $\tilde{\mu}_1,\ldots,\tilde{\mu}_J$.
\end{remark}

\subsection{A Lift to Multi-Marginal Optimal Transport}\label{app:multimarginal}

On the augmented space $\tilde{\Y}\coloneqq \Y\cup \{\dum\}$ equipped with metric $\tilde{d}_C$ in \eqref{eq:augmetric}, we define for $p\geq 1$ and $J\in\mathbb{N}$ a Borel barycenter application $\tilde{T}^{J,p}_C\colon \tilde{\Y}^J\to \tilde{\Y}$ that takes as input $(y_1,\ldots,y_J)\in\tilde{\Y}$ and outputs any minimizer $y\in\tilde{\Y}$ of the function
\begin{align*}
f(y)=\sum_{i=1}^J \tilde{d}_C^p(y_i,y).
\end{align*} 
Of particular interest to us is the barycentric application restricted to inputs from $\tilde{\X}$. However, we collect some of its key properties for general input $(y_1,\ldots,y_J)\in\tilde{\Y}^J$. For this, we define the index set
\begin{align*}
\BS(y_1,\ldots,y_J)\coloneqq \left\lbrace i\, \mid \, y_i = \dum,\, 1\leq i \leq J\right\rbrace.
\end{align*}
If clear from the context, then the dependence on $y_1,\dots,y_J$ is suppressed and the set is simply denoted as $\BS$.
\begin{lemma}\label{lem:borelapplication_prop}
Fix some parameter $C>0$ and consider the space $\tilde{\Y}$ with metric $\tilde{d}_C$ as defined in \eqref{eq:augmetric}. For points $(y_1,\ldots,y_J)\in\tilde{\Y}^J$ it holds that
\begin{itemize}
\item[(i)] $\tilde{T}^{J,p}_C(y_1,\dots,y_J)=\dum$ if and only if $\es{i \not \in \BS}{} \tilde{d}_C^p(y_i,y) \geq (J-2\lvert \BS\rvert )C^p/2$ for any $y\in \tilde{\Y}$. In particular, if strict inequality holds then $\tilde{T}^{J,p}_C(y_1,\dots,y_J)=\dum$ is unique.
\item[(ii)] If $2\lvert \BS \rvert \geq J$ then it holds $\tilde{T}^{J,p}(y_1,\dots,y_J)=\dum$ with uniqueness if $2\lvert \BS \rvert > J$.
\item[(iii)] If $\tilde{T}^{J,p}(y_1,\dots,y_J)\neq \dum$ then it holds
\[
\tilde{T}^{J,p}_C(y_1,\dots,y_J)=\uamin{y\in \mathcal{Y}} \es{i \not \in \BS}{} \tilde{d}^p_C(y_i,y).
\]
\item[(iv)] If $C>2^{\frac{1}{p}}\text{diam}(\mathcal{Y})$, then for any points $y_1,\ldots,y_J\in\Y$ with  $\lvert \BS \rvert=0$ it holds that $\tilde{T}^{J,p}_C(y_1,\dots,y_J)=T^{J,p}(y_1,\dots,y_J)$ where the latter one is defined with respect to the usual metric $d^p$ on $\Y$.
\end{itemize}
\end{lemma}
A proof of this result is provided in \Cref{sec:proofpac}. \Cref{lem:borelapplication_prop} allows to characterize the centroid sets of the augmented measures $\tilde{\mu}_1,\ldots,\tilde{\mu}_J$ defined as 
\begin{align}\label{eq:augmentedcentroid}
\begin{split}
    \tilde{\mathcal{C}}_{KR}(J,p,C):=\Big\{ y \in \ty \ \vert \ &y=\tilde{T}_C^{J,p}(x_1,\dots ,x_J), x_i \in \text{supp}(\tilde{\mu}_i);\\ &d^p(y,x_i)\leq C^p \ \forall \ x_i\neq \dum  \Big\}.
\end{split}
\end{align}

\begin{remark} \label{rem:truncatedapplication}
We point out that computing $\tilde{T}^{J,p}_C$ is in general a difficult optimisation problem. While for squared euclidean distance, computing the barycentric application simply amounts to taking the mean of the $x_i$, even on the non-augmented space, there are no closed form solutions available for most choices of distances and values of $p$. This problem is exacerbated by the truncation of the distance $\tilde{d}$ at $C^p$ \citep[as also pointed out in][]{muller2020metrics}, since it implies that disregarding a certain subset of points and just computing the barycenter with respect to the remaining $x_i$ might in fact be optimal. However, initially it is not clear which $x_i$ to choose, turning this into a difficult combinatorial problem. 
\end{remark}
Recall that for any measure $\mu$ its support is contained in $\X$ a subset of $\Y$. The augmented measure $\tilde{\mu}$ is extended by an additional support point at $\{\dum\}$. In particular, while the centroid set is a subset of $\tilde{\Y}$ it only depends on the support of the measures $\tilde{\mu}_i$ contained in $\tilde{\X}\coloneqq \X \cup\{\dum\}$.

\begin{corollary}\label{cor:centroid_sets}
For the centroid sets of the augmented measures $\tilde{\mu}_i\in\mathcal{M}_+(\tilde{\X})$ with $1\leq i \leq J$ it holds 
\begin{align*}
\tilde{\mathcal{C}}_{KR}(J,p,C) \subset \mathcal{C}_{KR}(J,p,C) \cup \{\dum\} \subset \mathcal{C}_{KR}(J,p) \cup \{\dum\}.
\end{align*}
\end{corollary}

\begin{proof}
The first inclusion follows by statements (i) and (iii) in \Cref{lem:borelapplication_prop} and the observation that $\lvert \BS \rvert=J-L$. The second by applying $\mathcal{C}_{KR}(J,p,C)\subset \mathcal{C}_{KR}(J,p)$.
\end{proof}
\begin{remark}
One could define $\mathcal{C}_{KR}(J,p,C)$ in terms of $\tilde{d}_C$ instead of $d$ to obtain equality in the first inclusion. Replacing $T^{L,p}$ by $\tilde{T}^{L,p}_C$ in the definition of the centroid set would not alter any of the related proofs and yield slightly sharper control on the support of $(p,C)$-barycenter. However, as we consider the given definition to be more intuitive, we omit this improvement in the statement of the theorem. 
\end{remark}
Let $\Pi(\tilde{\mu}_1,\dots ,\tilde{\mu}_J)$ be the set of measures on $\tilde{\Y}\times\ldots\times\tilde{\Y}$ whose $i$-th marginal is equal to $\tilde{\mu}_i$ for all $1\leq i \leq J$. We refer to the elements of this set as \emph{multi-couplings} of $\tilde{\mu}_1,\dots ,\tilde{\mu}_J$. For $p\geq 1$ define the \emph{augmented multi-marginal transport problem} as 
\begin{align}\label{eq:balancedmultimarginal}
\min_{\pi\in \Pi(\tilde{\mu}_1,\dots ,\tilde{\mu}_J)} \int_{\tilde{\mathcal{Y}}^J}c_{p,C}(y_1,\dots ,y_J)\pi(dy_1,\dots ,dy_J),
\end{align}
where 
\begin{align*}
c_{p,C}(y_1,\dots,y_J)\coloneqq\frac{1}{J}\easysum{i=1}{J}\tilde{d}^p_C\left(y_i,\tilde{T}_C^{J,p}(y_1,\dots,y_J)\right).
\end{align*}

The relation between the augmented multi-marginal transport formulation \eqref{eq:balancedmultimarginal} and the $(p,C)$-barycenter is as follows.

\begin{proposition}\label{prop:multimarginal}
Let $\mu_1,\dots ,\mu_J \in \msrX$ and $\tilde{\mu}_1,\ldots,\tilde{\mu}_J$ be their augmented counterparts. If $\pi\in \Pi(\tilde{\mu}_1,\dots ,\tilde{\mu}_J)$ is a solution to the augmented multi-marginal problem \eqref{eq:balancedmultimarginal}, then the measure $\mu^\star\coloneqq(\tilde{T}_C^{J,p} \# \pi)_{\vert \mathcal{Y}}\in\msrY$ is a $(p,C)$-barycenter of the measures $\mu_1,\dots ,\mu_J$, where $\tilde{T}_C^{J,p} \# \pi$ denotes the pushforward of $\pi$ under $\tilde{T}_C^{J,p}$.  Moreover, for every $(p,C)$-barycenter $\mu^\star$, there exists a solution $\pi$ to the augmented multi-marginal transport problem, such that 
\begin{align*}
\mu^\star+\left( \sum_{i=1}^J \mathbb{M}(\mu_i)-\mathbb{M}(\mu^\star)\right) \delta_{\dum}=\tilde{T}_C^{J,p} \# \pi.
\end{align*}
In particular, it holds that
\begin{align*}
\min_{\pi\in \Pi(\tilde{\mu}_1,\dots ,\tilde{\mu}_J)} \int_{\tilde{\mathcal{Y}}^J}c_{p,C}(y_1,\dots ,y_J)\pi(dy_1,\dots ,dy_J)=\inf_{\mu \in \msrY} F_{p,C}(\mu).
\end{align*}
\label{mm_eqv_expl}
\end{proposition}

The proof follows straightforwardly along the lines of related statements for the multi-marginal optimal transport problem (\citealp[Theorem 8]{le2017existence}; \citealp[Lemma 8]{masarotto2019procrustes} or \citealp[Proposition 3.1.2]{panaretos2020invitation}). This correspondence between the $(p,C)$-barycenter problem and a balanced multi-marginal optimal transport serves as one of the key components in the proof of \Cref{thm:bary_prop}.

\section{Computational Issues and Numerical Experiments}\label{sec:ComputationsSimulations}
We present approaches to compute the $(p,C)$-barycenter problem by solving related OT problems. Based on this, we investigate the performance of the Wasserstein and $(p,C)$-barycenters on multiple synthetic datasets. For reference, we also report on results for two related concepts of unbalanced barycenters (UBCs), namely the Gaussian-Hellinger-Kantorovich and Wasserstein-Fisher-Rao barycenter.
\subsection{Algorithms}\label{subsec:algos}
\Cref{thm:bary_prop} and \Cref{prop:multimarginal} both allow to pose the augmented problem (recall \Cref{sec:pac}) as a linear program and using \Cref{lem:fre_eq} one can obtain a solution to the original problem by solving the augmented one. Using any linear program solver this enables the direct computation of an exact solution of this problem. However, the number of variables in this approach scales as the size of $\mathcal{C}_{KR}(J,p,C)$ and hence it turns out to be infeasible already for relatively small instance sizes. To compute $(p,C)$-barycenters at larger scales we revisit iterative methods to solve the (balanced) Wasserstein barycenter problem and give instructions how to use modifications of them to compute $(p,C)$-barycenters. In particular, we detail a multi-scale method which solves successive fixed-support $(p,C)$-barycenter LPs on increasingly refined support sets. This provides a meta-framework to adjust state-of-the-art solvers for the Wasserstein barycenter for $(p,C)$-barycenter computations.\\
To construct the augmented problem we add the dummy point $\dum$ to the support of the $\mu_i$'s, while setting its distance to all other locations to be $C^p/2$. Note, that by \Cref{lem:transportC} and \Cref{lem:add_mass} the truncation of $\tilde{d}$ at $C^p$ can be omitted if $\mathbb{M}(\tilde{\mu}_i)>3\max_{i=1,\dots ,J} \mathbb{M}(\mu_i)$. If this is not the case, we can enforce it by adding additional mass at $\dum$ in all augmented measures without changing the optimal value. \\
\subsubsection{LP-Formulation for the $\mathbf{(p,C)}$-Barycenter}
\label{app:sec_LP}
Using property (i) from \Cref{thm:bary_prop}, we can rewrite the augmented $(p,C)$-barycenter problem as a linear program similarly to the usual $p$-Wasserstein barycenter problem \eqref{eq:otbarycenter}. However, compared to the latter one, we replace the standard centroid set $\mathcal{C}_W(J,p)$ from \eqref{eq:WassersteinCentroid}, by the centroid set $\tilde{\mathcal{C}}_{KR}(J,p,C)$ of the augmented measures from \eqref{eq:augmentedcentroid}. This yields
\begin{alignat*}{3}
\label{eq:linProgBary}
\underset{\pi^{(1)},\dots,\pi^{(J)},a}{\min} \quad & \frac{1}{J} \easysum{i=1}{J}\easysum{j=1}{\lvert \tilde{\mathcal{C}}_{KR}(J,p,C) \rvert}&&\easysum{k=1}{M_i}\pi^{(i)}_{jk}c_{jk}^i \\
\text{s.t.} \quad \easysum{k=1}{M_i}\pi^{(i)}_{jk}&=a_j, &&\forall \ i=1,\dots ,J, \forall j=1,\dots,\lvert \tilde{\mathcal{C}}_{KR}(J,p,C) \rvert, \\
\easysum{j=1}{\lvert \tilde{\mathcal{C}}_{KR}(J,p,C) \rvert}\pi^{(i)}_{jk}&=b^i_k, &&\forall \ i=1,\dots ,J , \forall k=1,\dots ,M_i, \\
\pi^{(i)}_{jk}&\geq 0  &&\forall i=1,\dots ,J,  \forall j=1,\dots ,\lvert \tilde{\mathcal{C}}_{KR}(J,p,C)\rvert , \\& &&\forall k=1,\dots ,M_i,
\end{alignat*}
where $M_i=\lvert \tilde{\mathcal{X}}_i \rvert$ is the cardinality of the support of the augmented measure $\tilde{\mu}_i$. Here, $c^i_{jk}$ denotes the distance between the $j$-th point of $\lvert \tilde{\mathcal{C}}_{KR}(J,p,C) \rvert$ and the $k$-th point in the support of $\tilde{mu}_i$, while $b^i$ is the vector of masses corresponding to $\tilde{\mu}_i$. For practical purposes it may be advantageous to solve the multi-marginal problem instead of the $(p,C)$-barycenter problem. This changes the number of variables from $\lvert \tilde{\mathcal{C}}_{KR}(J,p,C) \rvert (1+\sum_{i=1}^J M_i$) to $\prod_{i=1}^J M_i$ and the number of constraints from $J\lvert \tilde{\mathcal{C}}_{KR}(J,p,C) \rvert +\sum_{i=1}^J M_i$ to $\sum_{i=1}^J M_i$. Depending on the value of $C$, and hence the cardinality of $\tilde{\mathcal{C}}_{KR}(J,p,C)$, it is possible to pick the problem with the smaller complexity. \\
While this formulation is appealing for proving theoretical statements as provided in \Cref{thm:bary_prop}, it quickly becomes computationally infeasible even for small scale problems as the number of variables in the LP grows potentially as $\prod M_i$. However, it still enables exact computations of $(p,C)$-barycenters for small scale examples, which is currently impossible for general UBCs. Though, while there has been some recent advancement for the $2$-Wasserstein barycenter in special cases \citep{altschuler2021wasserstein} these LP-based algorithms ultimately do not scale to large instance sizes.
\subsection{Iterative Algorithms and the Multi-Scale Approach}
\begin{figure}
    \centering
    \includegraphics[width=\textwidth]{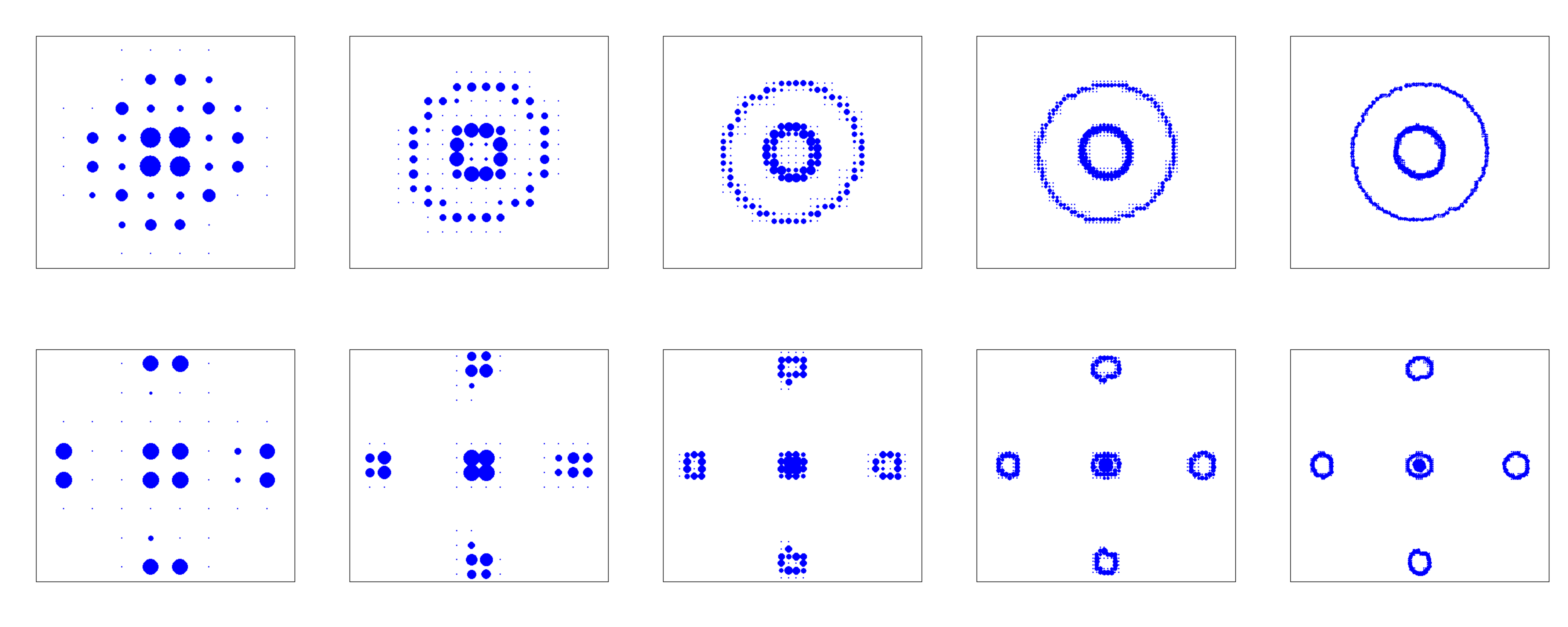}
    \caption{An illustration of the multi-scale approach on two different datasets. The fixed-support solutions are shown on grids of the sizes $8 \times 8$, $16\times 16$, $32\times 32$, $64\times 64$ and $128 \times 128$ increasing from left to right. \textcolor{mygreen}{The corresponding run-times on a single core of an Intel Core $i7$ $12700K$ in the first/second row were $2.5/5$ seconds, $14/16$ seconds, $145/42$ seconds, $13/3$ minutes and $143/22$ minutes.} \textbf{Top:} The dataset of nested ellipses from \Cref{fig:nested_ell}. \textbf{Bottom:} The dataset of ellipses with clustered support structure from \Cref{fig:cluster_ellipses}.}
    \label{fig:ms_illus}
\end{figure}
For the Wasserstein barycenter, iterative methods computing approximate barycenters, with a per iterations complexity only linear in the number of measures, enjoy great popularity. Most well known is the \emph{fixed-support Wasserstein barycenter} \citep{ge2019interior,lin2020fixed,xie2020fast} approach, aiming to find the best approximation of the barycenter on a pre-specified support set, for which a variety of methods is available. We utilise this fixed-support approach for the augmented $(p,C)$-barycenter problem by adding the dummy point $\dum$ to the given support and constructing the cost as described above. This yields a meta-framework which allows to employ fixed-support Wasserstein barycenter algorithms for fixed-support $(p,C)$-barycenter computation. One can also modify more general \emph{free support} methods \citep{cuturi2014fast,ge2019interior,luise2019sinkhorn}, which usually alternate between updating the support set of the barycenter and its weights on this set, to provide approximate $(p,C)$-barycenters. However, the necessary position updates usually explicitly or implicitly rely on being able to compute the barycentric application $\tilde{T}^{J,p}$ efficiently. Recalling \Cref{rem:truncatedapplication}, this is in general not tractable for the augmented problem, which severely hinders the use of these approaches. Thus, it is tempting to avoid these issues by approximating $\Y$ with a large finite space, i.e., by taking a grid of high-resolution, and solving the fixed support $(p,C)$-barycenter problem on this set. However, solving the fixed-support problem on this large space requires significant computational effort. We advovate an alternative by adapting the ideas of multi-scale methods for the Wasserstein distance/barycenter \citep{merigot2011multiscale,gerber2017multiscale,schmitzer2019stabilized} to the $(p,C)$-barycenter setting. The idea of this approach is to start with a coarse version of the problem and then successively solve refined problems, while using the knowledge of the coarse solution to reduce the complexity of the finer ones. \\
Thus, we initialise the support set of the barycenter as a fixed grid of size $K_1 \times \dots \times K_d$ in $\mathbb{R}^d$. In the $j$-th step of the algorithm, after solving the fixed-support problem, we remove the grid points which have zero mass and replace the remaining ones with its $2^d$ closest points in a refined version of the original grid of size $2^jK_1 \times \dots \times 2^jK_d$. This can be understood as solving the fixed-support problem on successively finer grids, while  incorporating information provided by having already solved a coarser solution of the problem. We terminate the method once a pre-specified resolution has been reached. This allows to obtain fixed-support approximation of the $(p,C)$-barycenter on fine grids without having to optimise over the full support set. \\
We point out that this approach, while inspired by multi-scale approaches is more closely related to the formerly mentioned free-support methods. As such it does in general not yield a globally optimal fixed-support $(p,C)$-barycenter at the finest resolution. Instead it converges to a local minimum of the unbalanced Fr\'echet functional depending on the resolution of the initial grid. This is a common problem among alternating procedures for the free-support barycenter problem and can be attributed to the fact that the Fr\'echet functional is non-convex in the support locations of the measures. However, we stress that with this approach we observe reasonable approximations of the $(p,C)$-barycenter while avoiding the inherent problems of generalising usual position update procedures discussed above. In particular, we do not have to solve the $\tilde{T^{J,p}_C}$ barycenter problem at any point. Additionally, we note that the initial grid size should be chosen at least fine enough that the distance between two adjacent grid points is smaller than $C$. Otherwise it is possible that support points lying between two grid points, having distance larger $C$ to both, are not accounted for. For a visual illustration of the algorithm we refer to \Cref{fig:ms_illus}.
\subsection{Synthetic Data Simulations}
We test the performance of the $(p,C)$-barycenter as a data analytic tool compared to the usual $p$-Wasserstein barycenter on a multitude of datasets. We base our computations on the MAAIPM method \citep{ge2019interior}, which allows for high-precision approximations of barycenters up to moderate data sizes. The algorithm has been deployed to solve the fixed-support $(p,C)$-barycenter problems arising in the multi-scale method detailed above. \textcolor{myred}{For all experiments, the initial grid size as been set to $16\times 16$ and the refinement is terminated at a gridsize of $128\times 128$. Values below $10^{-5}$ have been considered as zero for the purposes of grid refinement. All experiments have been carried out on a single core of an Intel Core $i7$ $12700K$.} Implementations of our used method and some alternatives can be found as part of the R-package \emph{WSGeometry} (on CRAN).
\begin{figure}
\centering
\includegraphics[width=\textwidth]{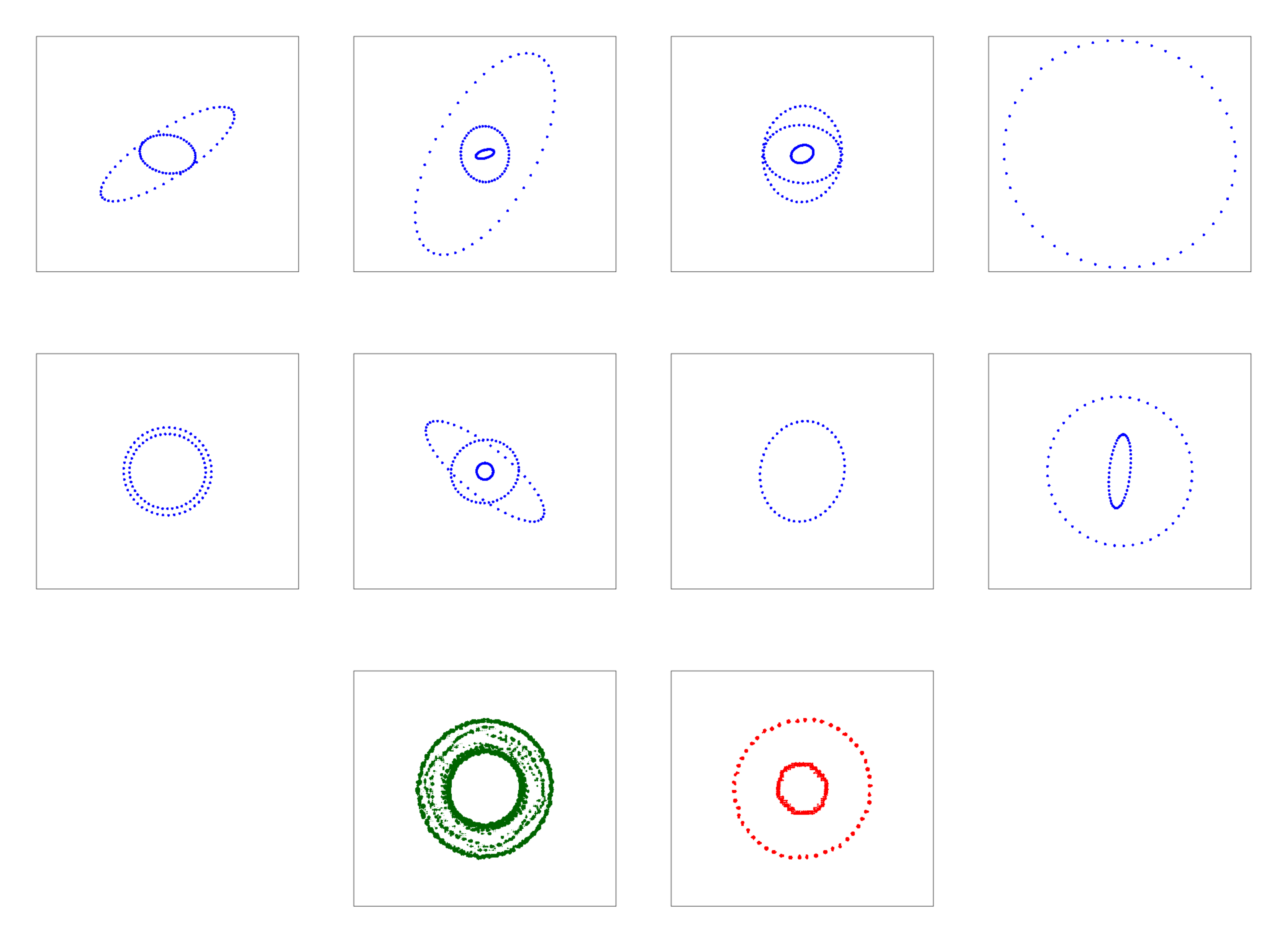}
\caption{An excerpt of a dataset of $N=100$ discretized ellipses. Each measure contains between $1$ and $3$ ellipses with equal probability. Each ellipse consists of $50$ points with mass $1$ in $[0,1]^2$. \textbf{Left:} In darkgreen the $2$-Wasserstein barycenter, where all measures are normalized to be probability measures \textcolor{myred}{(runtime about 8 hours)}. \textbf{Right:} In red, the $(2,1.5)$-barycenter \textcolor{myred}{(runtime about 30 minutes)}.
\label{fig:nested_ell}}
\end{figure}
\subsection*{Mismatched Shapes}
This first set of examples mainly serves as starting point to illustrate improved performance of the $(p,C)$-barycenter compared to the $p$-Wasserstein barycenter. A prototypical benchmark for the $p$-Wasserstein barycenter are two nested ellipses as popularized in \cite{cuturi2014fast}. For our example of nested ellipses, we assume that the support of each measure consists of nested ellipses, but the number of ellipses varies between the individual underlying measures. Specifically, we assume that for each $\mu_i$ the number of ellipses is uniformly random in $\{1,2,3\}$ and that each ellipse is discretised onto $M$ support points with unit mass, respectively.  This can be seen in \Cref{fig:nested_ell}. We observe that while the $p$-Wasserstein barycenter recovers the elliptic shape of the underlying measures, it fails to produce distinct ellipses and instead produces something akin to a ring. In contrast, the $(p,C)$-barycenter yields two distinct ellipses, which coincides with the expected number of ellipses in one of the measures. This aligns well with intuition that the $(p,C)$-barycenter will simply disregard any additional structures which are not present in a sufficient amount of underlying measures. In contrast, the $p$-Wasserstein barycenter does not allow for this flexibility which enforces additional support points.   
\begin{figure}
\centering
\includegraphics[width=\textwidth]{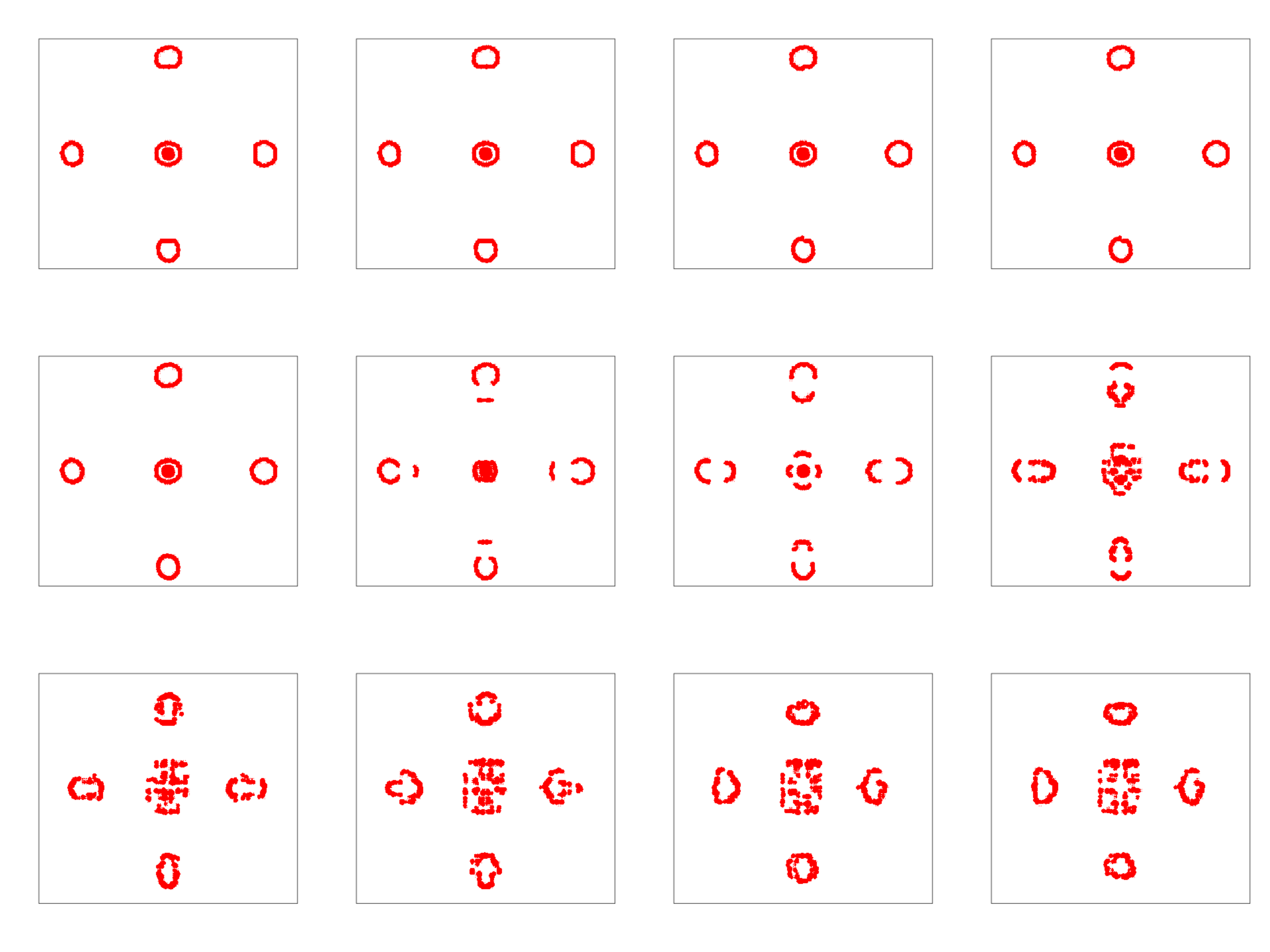}
\caption{The $(2,C)$-barycenters for the measures in \Cref{fig:cluster_ellipses} for different values of $C$. From top-left to bottom-right the values of $C/\textcolor{myred}{\text{runtime in minutes}}$ are equal to \textcolor{myred}{$0.1/28$, $0.15/28$, $0.2/26$, $0.25/28$, $0.275/27$, $0.3/33$, $0.35/59$, $0.4/132$, $0.45/156$, $0.5/160$, $0.55/160$, $0.6/155$}, respectively. \label{fig:cluster_ellipses_ccurve}}
\end{figure}

\subsection*{Local Scale Cluster Detection}
Recall the setting of \Cref{fig:cluster_ellipses}. In the following class of examples, we are interested in datasets which possesses a natural cluster structure. Let $B_1,\dots ,B_R\subset \mathbb{R}^D$ be convex, disjoint sets and assume that $\text{supp}(\mu_i)\subset \cup_{r=1}^{R}B_r$ for all $i=1,\dots ,J$. If the diameter of all $B_r$ is bounded from above by $C$ and that the distance between each two $B_r,B_s$ is at least $2^{1/p} C$, then \Cref{lem:cluster} guarentees that the $(p,C)$-barycenter detects all of the $R$ clusters in which at least $J/2$ measures have positive mass. In particular, by \Cref{thm:bary_prop} (v) the $(p,C)$-barycenter will have mass in all of those clusters. Intuitively, this setting is reasonable if, for instance, it is already known that any interactions between support points of different measures are limited to scales below a certain threshold, which should then be chosen as $C$. The lower bound on the inter-cluster distance ensures that any pair of two clusters is well-separated, ensuring that it is always possible to distinguish between two different clusters, as they can not be arbitrarily close to each other.\\
In \Cref{fig:cluster_ellipses} the $p$-Wasserstein barycenter completely fails to capture the geometric data structure. Most of its mass is between the clusters and the outer clusters have nearly no mass. Moreover, the elliptic structure within each cluster is clearly not captured. In contrast, the $(p,C)$-barycenter not only captures all clusters, it also distinguishes between the difference in intensity (expected number of ellipses) in the clusters, matching the theoretical guarantees of \Cref{lem:cluster}. We stress that for this example the choice of $C$ is of particular importance. If we choose $C$ too large, the $(p,C)$-barycenter will fail to recover the data's support structure (for an illustration of the $(p,C)$-barycenter in this example over a range of values of $C$ see \Cref{fig:cluster_ellipses_ccurve}). Consequently, it is crucial to choose $C$ appropriately. In this example, the barycenter appears to be stable and detect all clusters for $C\in [0.1,0.275]$. Notably, if the locations of the clusters are already known, this setting also allows for parallel computations of the $(p,C)$-barycenter, where the problems are solved separately on each cluster and recombined at the end (\Cref{lem:cluster}). 
\subsection*{Randomly distorted Measures}
In a statistical context it is important to investigate the stability of the $(p,C)$-barycenter under random distortions. We fix a reference measure $\mu_0$ on $\mathbb{R}^d$ and generate a set of measures by random modifications of $\mu_0$. We then attempt to recover $\mu_0$ by computing the $p$-Wasserstein and $(p,C)$-barycenter of these measures, respectively. \\
In the following, let $B(p)$ denote a Bernoulli random variable with mean $p$, $Poi(\lambda)$ a Poisson distribution with mean $\lambda$ and $U[a,b]$ a uniform distribution on $[a,b]$. We generate $\mu_1,\dots ,\mu_J$ as follows: \\
For $i=1,\dots ,J$ initialise $\mu_i=\mu_0$, then succesively modify $\mu_i$ based on the four following steps.
\begin{itemize}
\item[(i)] \textbf{Point Deletion:} Fix $p_{del}\in [0,1]$ and $\lambda_{del} \in \mathbb{R}_+$. We draw a Ber($p_{del})$ random variable. If it takes the value $1$, then we draw $D\sim Poi(\lambda_{del})$ and select $\min(D,\lvert \text{supp}(\mu_0)\rvert)$ points in the support of $\mu_0$ uniformly by drawing without replacement. These points (and their mass) are not contained in $\mu_i$, since they have been deleted. 
\item[(ii)] \textbf{Point Addition:} We fix parameters $p_{add}\in [0,1],\lambda_{add}\in \mathbb{R}_+, m_{add}\in \mathbb{R}^2, \sigma_{add}\in \mathbb{R}^{2\times 2},u_0,u_1\in \mathbb{R}$. Draw a Ber($p_{add}$) random variable. If it takes the value $1$, draw a Poi($\lambda_{add}$) random variable $\alpha$. Then, generate $\alpha$ random variables following a normal distribution with mean $m_{add}$ and covariance matrix $\sigma_{add}$. Add these support points to $\mu_i$, where the weight of each of these points is determined by independent $U[u_0,u_1]$ random variables. 
\item[(iii)] \textbf{Position Change:} Fix parameters $a_1,a_2,b_1,b_2\in \mathbb{R}$ with $a_1 \leq b_1$ and $a_2\leq b_2$. For each $x_0$ in the support of $\mu_i$, we draw a $U([a_1,b_1]\times [a_2,b_2])$ random variable and shift the position of $x_0$ by it.
\item[(iv)] \textbf{Weight Change:} Fix parameters $l,u\in \mathbb{R}$ with $l\leq u$. For each support point $x_0$ of $\mu_0$ with weight $w_0$, we draw a $U[l,u]$ random variable $U$ and change the weight of $x_0$ in $\mu_i$ to be $w_0+U$. 
\end{itemize}
An example of this setting can be seen in \Cref{fig:noise_ell}. Comparing the two barycenters displayed there to the original measure reveals that, while the rough shape of the $2$-Wasserstein barycenter is correct, its mass is spread out over a larger area and it has a significantly larger number of support points. Since all measures have been normalised, we have also lost all information on the mass of $\mu_0$. Contrary to that, the $(p,C)$-barycenter retrieves the original measures recovering the location and number of the of support points closely. Additionally, it also has a mass which only deviates from the original mass by about $0.23\%$. If one is only interested in recovering the general shape of the data, both approaches provide comparable performance. However, if the measures total mass and more detailed support structure are of importance the $(p,C)$-barycenter appears to be preferable.
\begin{figure}
\centering
\includegraphics[width=\textwidth]{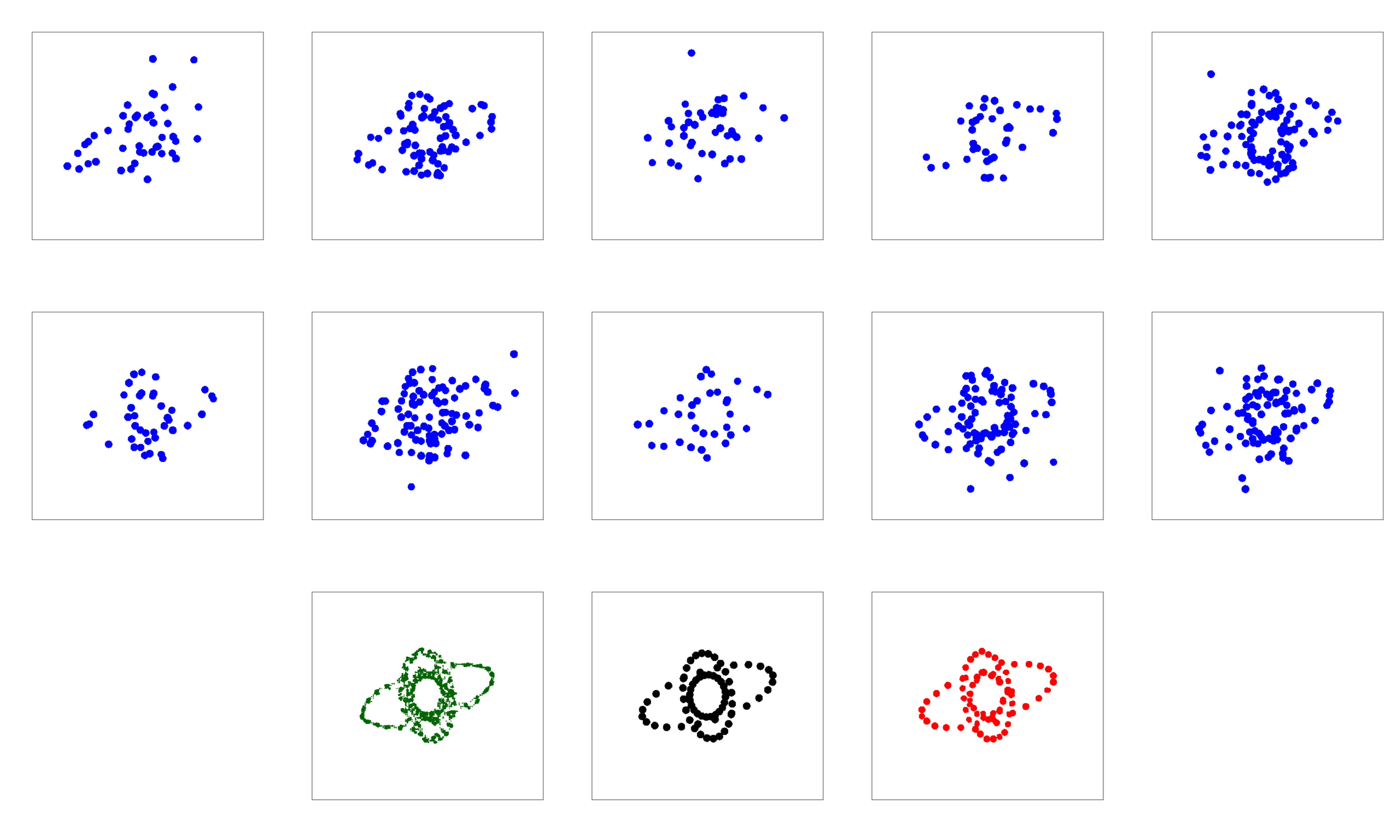}
\caption{An excerpt from a dataset of $N=100$ noisy nested ellipses supported in $[0,1]^2$. The parameters are $p_{del}=1/3$, $\lambda_{del}=75$, $p_{add}=1/3$, $\lambda_{add}=25$, $m_{add}=(0.5,0.5)^T$, $\sigma_{add}=0.15I_{2}$, $u_0=0.9$, $u_1=1.1$, $a_1=a_2=-0.025$, $b_1=b_2=0.025$, $l=u=0.1$. \textbf{Left:} The $2$-Wasserstein barycenter in dark green \textcolor{myred}{(runtime about 4 hours)}. \textbf{Center:} The original measure $\mu_0$ (black). \textbf{Right:} The $(2,1.5)$-barycenter in red \textcolor{myred}{(runtime about 20 minutes)}. \label{fig:noise_ell}}
\end{figure}
\subsection*{Total Mass Intensity}

\begin{figure}
    \centering
    \includegraphics[width=\textwidth]{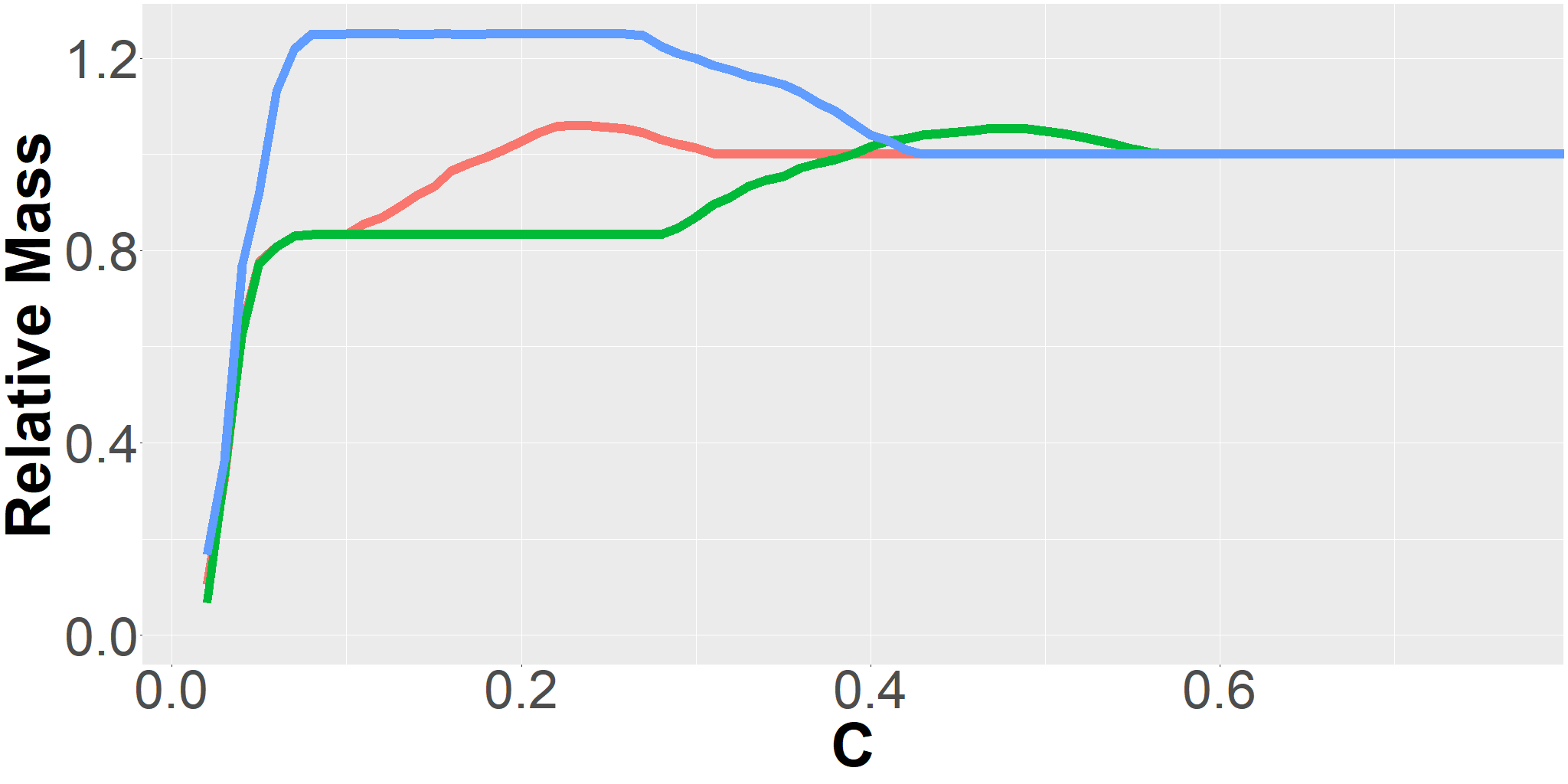}
    \caption{\textcolor{mygreen}{The mass of a $(p,C)$-barycenter for three sets of measures relative to the median of the total mass intensities of these measures. The green line corresponds to $J=25$ measures from the same class as considered in \Cref{fig:cluster_ellipses}. The red line corresponds to the same measures where the four outer clusters have been moved closer to the central one, such that their distance has been halved. The blue line corresponds to $J=5$ measures with the same cluster structure as in \Cref{fig:cluster_ellipses}, where the total number of ellipses in all clusters is fixed to be equal to four for all $J$ measures.}}
    \label{fig:barymass}
\end{figure}

\textcolor{mygreen}{While the $p$-Wasserstein barycenter of $J$ probability measures has mass one, the mass of the $(p,C)$-barycenter depends on $C$ as well as the geometry of the measures $\mu_1,\dots,\mu_J \in \msrX$. Exact values for the mass of a $(p,C)$-barycenter without detailed computations, are only available in the limiting scenarios where $C$ is extremely small or large relative to the other distances in $\X$. For the former, we know by \Cref{thm:bary_prop} $(v)$ that the barycenter has mass zero for disjoint measures and for the latter, \Cref{thm:bary_prop} $(vi)$ yields that there exists a $(p,C)$-barycenter with total mass intensity equal to the median of $\mathbb{M}(\mu_1),\dots,\mathbb{M}(\mu_J)$. For intermediate values of $C$, \Cref{thm:bary_prop} $(i)$ yields the upper bound by $2J^{-1}\sum_{i=1}^J\mathbb{M}(\mu_i)$. To highlight some possible behaviours of the total mass intensity of $(p,C)$-barycenter we consider three specific examples in \Cref{fig:barymass}. We note that in all three cases at about $C=0.6$ the mass of the barycenters is at the median of their respective $\mu_1,\dots,\mu_J$ and does no longer change with increasing $C$. This is significantly smaller than the requirement in \Cref{thm:bary_prop} $(vi)$, which underlines the fact that while in the worst case, this lower bound is sharp, in many examples the total mass of the $(p,C)$-barycenter stabilises significantly earlier. Moreover, none of the three curves is monotone. Instead the total mass of the barycenter is increasing up to a certain point, after which it decreases until it reaches the median of the masses. This makes intuitive sense, as the measures are disjoint, thus for small $C$ the barycenter is empty and starts to grow in mass quickly as the points within the clusters can be matched. In particular, the differences in intensity between clusters might lead to a total mass over the median $\mathbb{M}(\mu_1),\dots,\mathbb{M}(\mu_J)$, as by \Cref{lem:cluster} the total mass intensity of the $(p,C)$-barycenter is $\sum_{r=1}^R med(\mathbb{M}(\mu_{1_{\vert B_r}}),\dots,\mathbb{M}(\mu_{J_{\vert B_r}}) $, where $B_1,\dots,B_5$ denote the respective cluster locations. For larger $C$ these clusters start to merge and support points between the clusters reduce the total mass. In particular, these points can be seen clearly in the plot. Up until about $C=0.1$, which is the cluster size, the mass of the barycenters rises sharply, before stabilising until the intercluster distance is reached. This is about $0.3$ for the green and blue lines and about $0.15$ for the red line (since the measures in this example are generated by halving the intercluster distance from the green one). This behaviour highlights the sensitivity of the mass of the $(p,C)$-barycenter to the geometry of the measures. It is therefore impossible to infer the total mass of the $(p,C)$-barycenter from the magnitude of $C$ alone without accounting for the specific measures. However, analysing the structural properties of the support sets of the measures might provide a good indication at what values of $C$ changes in drastic behaviour of the total mass are to be expected. }
\subsection{Comparison with Related Unbalanced Barycenter Concepts}
We compare the $(p,C)$-barycenter with two alternative UBC approaches. \\
\textbf{The Gaussian-Hellinger-Kantorovich Barycenter:} This example falls in the general framework of optimal entropy transport problems. Measuring deviation between a feasible solution and the input marginals is carried out via the \textit{Kullback-Leibler divergence} defined for $\mu\ll\nu$ \footnote{A measure $\mu\in \msrX$ is said to be \emph{absolutely continuous} (denoted $\mu \ll \nu$) with respect to another measure $\nu\in \msrX$ if $\nu(A)=0$ implies $\mu(A)=0$ for any measurable set $A$.} as
\[
KL(\mu,\nu)=\es{x\in X}{} \mu(x)\log \left(\frac{\mu(x)}{\nu(x)}\right). 
\]
If $\mu\not\ll\nu$ the value of $KL$ is set to be $+\infty$. For a parameter $\lambda>0$, the \textit{Gaussian-Hellinger-Kantorovich Distance} \citep{liero2018optimal} is defined as 
\[
GHK_{\lambda}(\mu,\nu)= \min_{\pi\in \mathcal{M}_+(\X\times\X)} \sum\limits_{x,x^\prime \in \X}d^2(x,x^\prime)\pi(x,x^\prime) + \lambda KL(\mu,\pi_1) + \lambda KL(\nu,\pi_2),
\] 
where $\pi_1$ and $\pi_2$ denote the respective marginals of $\pi$. The $GHK_\lambda$ barycenter is defined as
\[
\argmin_{\mu\in\mathcal{M}_+(\Y)} \esn GHK_\lambda(\mu_i,\mu).
\]\\
\textbf{The Hellinger-Kantorovich Barycenter:} The Hellinger-Kantorovich distance, also known as Wasserstein-Fisher-Rao distance \citep{liero2018optimal,chizat2018interpolating}, is closely related to the Gaussian-Hellinger-Kantorovich distance. For fixed parameter $\sigma \in (0,\pi/2]$, referred to as the \textit{cut-locus}, it is defined as 
\begin{align*}
    HK_{\sigma}(\mu,\nu)= \min_{\pi\in \mathcal{M}_+(\X\times\X)} &\sum\limits_{x,x^\prime \in \X}(-\log (\cos_\sigma^2(d(x,y)))\pi(x,x^\prime) \\ + &KL(\mu,\pi_1) + KL(\nu,\pi_2),
\end{align*}
where $cos_\sigma: z\mapsto \cos(\min (z,\sigma))$. For a fixed cut-off locus $\sigma$, the $HK_\sigma$ barycenter is defined as
\[
\argmin_{\mu\in\mathcal{M}_+(\Y)} \esn HK_\sigma(\mu_i,\mu).
\]
\textbf{Comparing the barycenters:} As the resulting barycenters vary significantly in all three cases, depending on the parameters $C,\lambda,\sigma$, we compare their behaviour upon change of parameter.
\begin{figure}
\includegraphics[width=\textwidth]{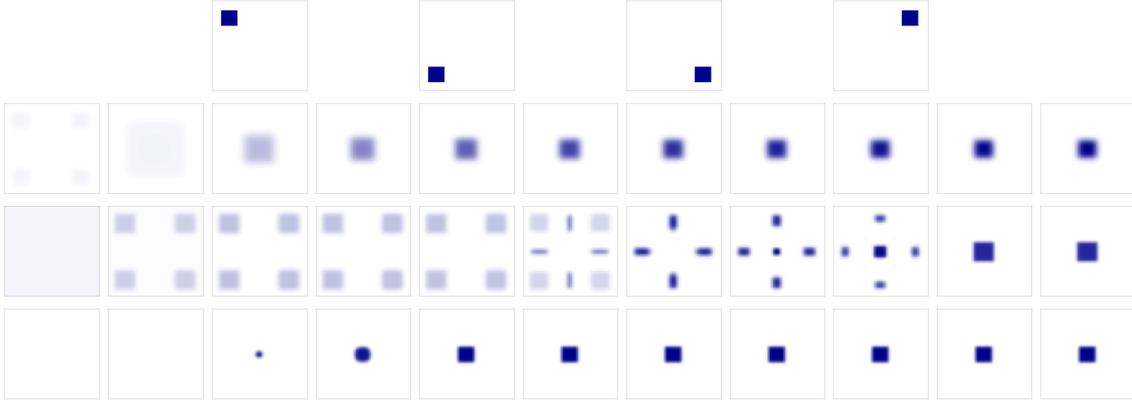}
\caption{Comparison of the three unbalanced barycenters when varying their parameter. All measures are supported on an equidistant $64\times 64$ grid in $[0,1]^2$. \textbf{First row:} The four underlying measures. \textbf{Second row:} The Gaussian-Hellinger-Kantorovich barycenter for $\lambda=0.01,0.15,\dots,1.83,1.97$. \textbf{Third row:} The Hellinger-Kantorovich barycenter for $\sigma=0.01,0.08,\dots,0.64,0.71$. \textbf{Fourth row:} The KR barycenter for $C=0.01,0.08,\dots,0.64,0.71$. \label{fig:unb_comp}}
\end{figure}
As a simple example, we consider four measures supported on subsets of a grid on $[0,1]^2$, displayed in \Cref{fig:unb_comp}. To ensure fair comparison, we deploy the same method based on the general scaling method \citep{chizat2018scaling} to approximate the UBC in all three cases. However, we point out that this implies disregarding the ambient space and instead taking the minimum over all positive measures supported on a prespecified grid in $[0,1]^2$. \\ 
For high parameter values all three approaches yield similar results. This is, of course, to be expected, since these distances interpolate between $p$-Wasserstein distance and total variation/Kullback-Leibler distance and large parameters correspond to a setting being close to the Wasserstein distance. The KR barycenter has mass zero for small choice of $C$ by \Cref{thm:bary_prop} (iv), since the four measures have disjoint support. After reaching a threshold of $C\approx 0.1$, the mass in the $(2,C)$-barycenter starts to increase as mass is added in the center of the unit square until at $C\approx 0.3$ the mass of an individual data measure is reached.\\
For small $\lambda$ the $GHK_\lambda$ barycenter has small mass and its support is close to that of a linear mean of the four measures, though the total mass intensity is significantly lower than for the original measures. With increasing $\lambda$ the mass starts to increase and to smear into the middle of the unit square, until a large square, encompassing all four data supports, is formed. After this point increasing $\lambda$ causes the square to contract while its mass increases. Finally, we approach a single square at roughly the same size as the squares in the underlying measures for large $\lambda$. \\
The $HK_\sigma$ barycenter is close to a linear mean of the four measures for small cut-off. Increasing $\sigma$ initially reduces the mass at each of the square locations. At a threshold of $\sigma\approx0.34$, we observe a change, where part of the mass is moved vertically or horizontally to the mid points between the squares in a rectangular shape. Until $\sigma\approx 0.43$ all mass is shifted to these "middle-rectangles", at which point a second shift occurs, where the mass from these rectangles starts to move towards a square in the center. At $\sigma \approx 0.6$, all mass has been shifted towards a square in the center and there is no further change in the HK barycenter, when increasing $\sigma$.\\
\begin{figure}
    \centering
    \includegraphics[width=0.95\linewidth]{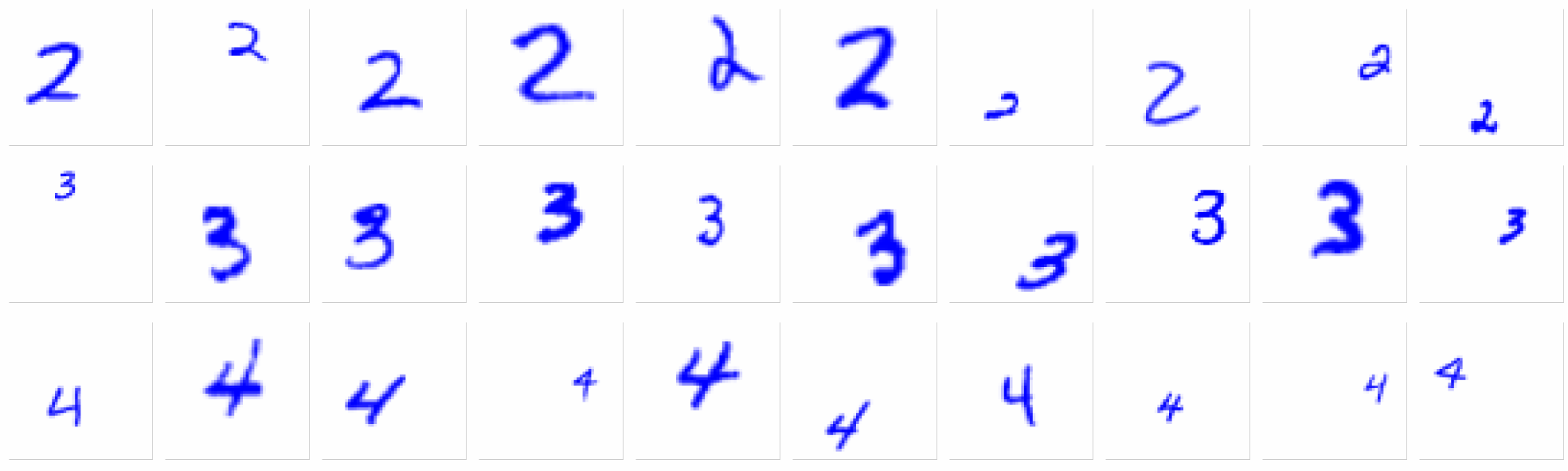}
    \caption{\textcolor{mygreen}{Images displaying the underlying measures used for barycenter computation in \Cref{fig:mnist}. Each row corresponds to a dataset of ten elements of the classical MNIST dataset which have been randomly rescaled and shifted within a $50\times 50$ grid in $[0,1]^2$. Their total mass intensities have not been normalised.}}
    \label{fig:mnist_data}
\end{figure}
\begin{figure}
  \centering
  \subfloat[][]{\includegraphics[width=0.95\linewidth]{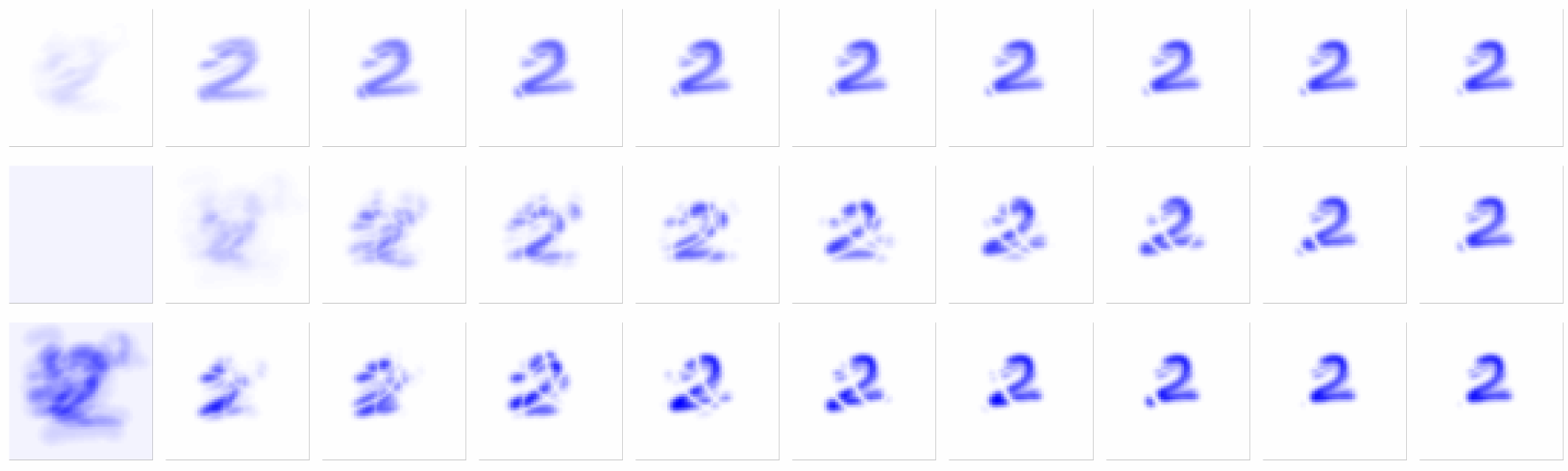}} \hspace{0.05\linewidth}
    \subfloat[][]{\includegraphics[width=0.95\linewidth]{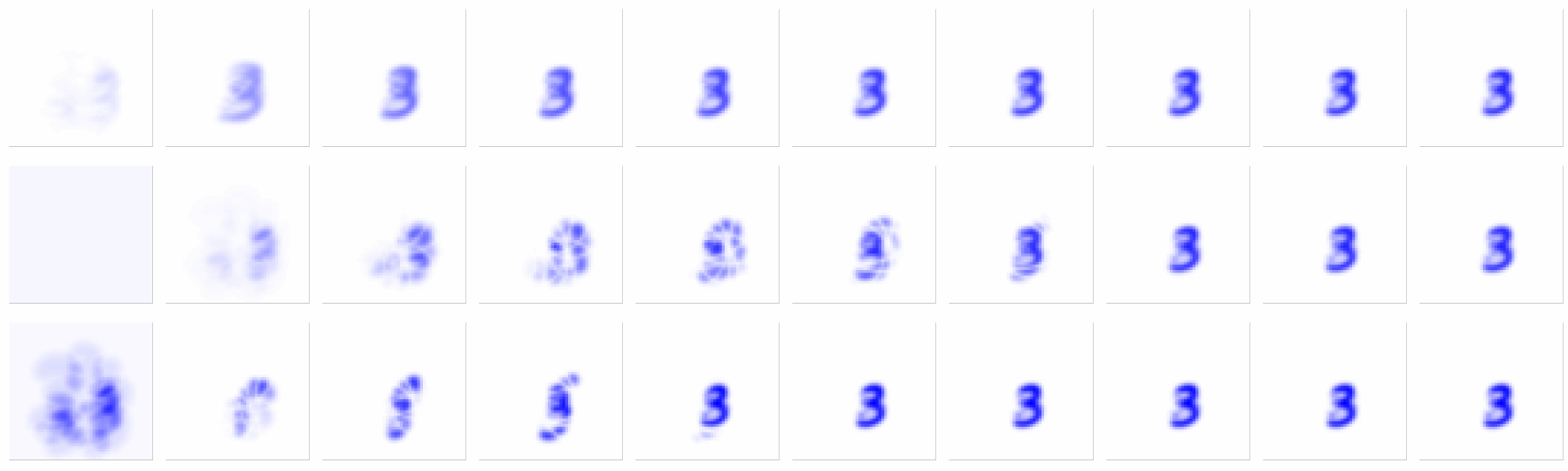}} \hspace{0.05\linewidth}
  \subfloat[][]{\includegraphics[width=0.95\linewidth]{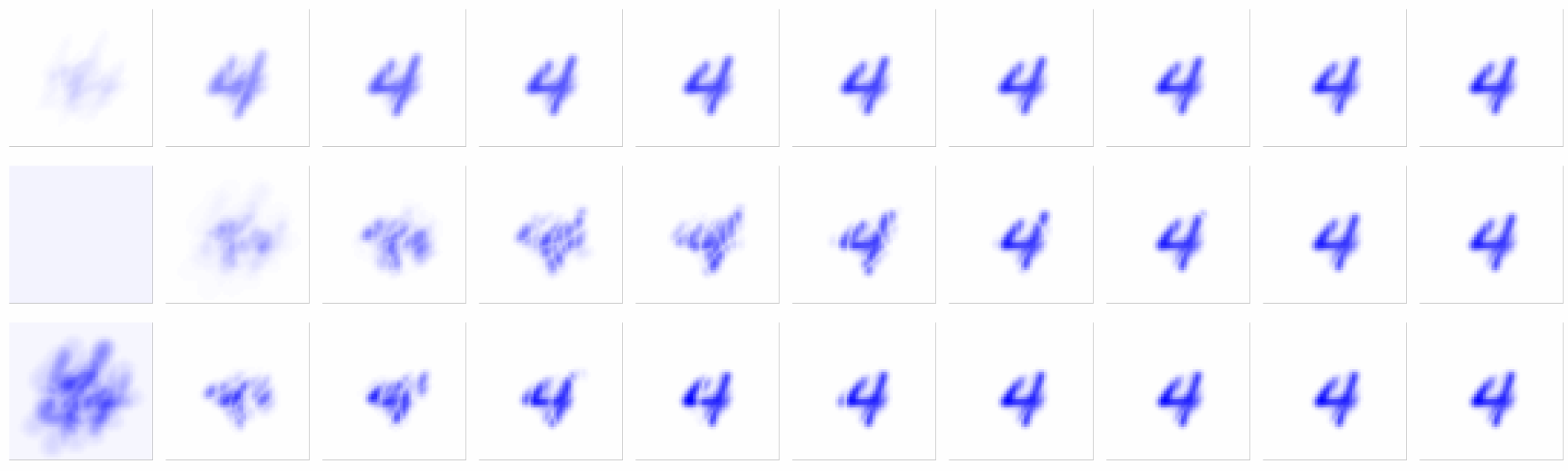}}%
  \caption{\textcolor{mygreen}{Comparison of the three unbalanced barycenters when varying their parameter. The set of underlying measures for (a) is the first row of \Cref{fig:mnist_data}. For (b) it is the second for (c) the third. For each class of examples the three different UOT barycenter models are considered in different rows: \textbf{First row:} The Gaussian-Hellinger-Kantorovich barycenter for $\lambda= 0.01, 0.12, 0.23,\dots,1$. \textbf{Second row:} The Hellinger-Kantorovich barycenter for $\sigma=0.01,0.08,\dots,0.64$. \textbf{Third row:} The KR barycenter for $C=0.1,0.2,\dots,1$.} \label{fig:mnist}}%
\end{figure}
\textcolor{mygreen}{Additionally, we consider \Cref{fig:mnist}, where the three unbalanced barycenter models are compared on three exemplary classes based on the MNIST dataset. Here, the original $28\times 28$ images have been rescaled to sizes between $14\times 14$ and $42\times 42$ and embedded in a random subgrid of a $50\times 50$ image. In this setting, there is a notable distinction between the GHK barycenter and the KR and HK barycenters. While for the former, the overall shape is recovered even for small parameter values, the latter two barycenters produce unstructured results for small parameters. The GHK distance is not constructed to have a maximal transport distance comparable to the impact of $C$ or $\sigma$ in the other two cases, which allows to transport across larger distance and recover the correct shape for smaller values of $\lambda$. However, the mass of the GHK barycenter is significantly smaller than that of the original measures for small values of $\lambda$ and only increases to the correct magnitude for larger penalty values. The HK and KR barycenters consist of fragments of the final shape which move towards a joint location for increasing parameters. For large penalties all three models are nearly identical and display the corresponding number correctly. This makes sense, as in this setting the minimisation in any individual term of the $(p,C)$-Fr\'echet functional is driven by minimising an OT term. We point out that for the $(p,C)$-barycenter this regime is guaranteed to be reached by choosing $C$ larger than the diameter of the space, while for the other two models the suitable parameter choice for this example is ambiguous without actually computing the result for specific values.}\\

Overall, for large parameter values all considered UBCs perform similarly. In small parameter regimes we observe significant differences. This difference in behavior is to be expected as the dependence of the UOT models on their parameters varies significantly. One key advantage of the KR barycenter is that its connection between the choice of $C$ and the properties of the resulting barycenter is immediate and intuitive. While the cut-off locus $\sigma$ for the HK barycenter fulfils a similar role, imposing control at the maximum scale at which transport does occur, the consequences of changing $\sigma$ from one value to another are far less immediate due to the involved structure of the cost functional in this setting. Similarly to the KR barycenter, it is worth noticing that the HK barycenter does allow for mass at locations given by centroids of support points of $L<N$ measures. Though, while for the KRD a feature of the underlying measures is only contained in the barycenter if it is present in more than $L=N/2$ measures, the HK barycenter also allows for mass at locations constructed from less support points. Thus, the HK barycenter is prone to being more susceptible to errors due to noise within the data. Compared to the other two choices, the parameter $\lambda$ of the GHK barycenter does appear to have less interpretation, with the only clear connection being that increasing $\lambda$ increases the mass of the GHK barycenter. There does also not appear to be any well-founded method how to approach the choice of $\lambda$ for a given dataset.

\section*{Acknowledgments}
F.\  Heinemann and M.\ Klatt gratefully acknowledge support from the DFG Research Training Group 2088 \textit{Discovering structure in complex data: Statistics meets optimization and inverse problems}. A.\ Munk gratefully acknowledges support from the DFG CRC 1456 \textit{Mathematics of the Experiment A04, C06} and the Cluster of Excellence 2067 MBExC \textit{Multiscale bioimaging--from molecular machines to networks of excitable cells}. We kindly thank two anonymous referees for their helpful comments and suggestions.

\bibliographystyle{abbrvnat}
\bibliography{unbalanced_barycenter}{}

\begin{appendix}

\section{Proofs}

\subsection{Proofs of \Cref{sec:pac}}\label{sec:proofpac}
\begin{proof}[Proof of \Cref{lem:fre_eq}]
Let $\mu \in \msrY$ be such that $\mathbb{M}(\mu)\leq \sum_{i=1}^J \mathbb{M}(\mu_i)$. Then 
\begin{align*}
F_{p,C}(\mu)\overset{(i)}{=}&\frac{1}{J}\sum_{i=1}^J \tilde{\text{OT}}_{\tilde{d}^p_C}^p\left(\mu+ \mathbb{M}(\mu_i)\delta_{\dum},\mu_i+\mathbb{M}(\mu)\delta_{\dum}\right)\\
\overset{(ii)}{=}&\frac{1}{J}\sum_{i=1}^J \tilde{\text{OT}}_{\tilde{d}^p_C}^p\left(\mu+\left( \sum_{i=1}^J \mathbb{M}(\mu_i)-\mathbb{M}(\mu)\right) \delta_{\dum},\tilde{\mu}_i\right)\\
=&\tilde{F}_{p,C}\left(\mu+\left( \sum_{i=1}^J \mathbb{M}(\mu_i)-\mathbb{M}(\mu)\right) \delta_{\dum}\right),
\end{align*}
where $(i)$ follows from the lift to an OT problem (\Cref{app:lifttoOT}) and $(ii)$ follows from \Cref{lem:add_mass} by adding mass $\sum_{j\neq i} \mathbb{M}(\mu_j)-\mathbb{M}(\mu)$ at $\dum$. We then have that 
\begin{align*}
\min_{\substack{\mu\in\msrY\\\mathbb{M}(\mu)\leq \sum_{i=1}^J\mathbb{M}(\mu_i)}}F_{p,C}(\mu)
&=\min_{\substack{\mu\in\msrY\\\mathbb{M}(\mu)\leq \sum_{i=1}^J\mathbb{M}(\mu_i)}} \tilde{F}_{p,C}\left(\mu+\left( \sum_{i=1}^J \mathbb{M}(\mu_i)-\mathbb{M}(\mu)\right) \delta_{\dum}\right)\\
&\geq \min_{\substack{\mu\in\mathcal{M}_+(\tilde{\Y})\\\mathbb{M}(\mu)=\sum_{i=1}^J \mathbb{M}(\mu_i)}} \tilde{F}_{p,C}(\mu)
\end{align*}
and 
\begin{align*}
\min_{\substack{\mu\in\mathcal{M}_+(\tilde{\Y})\\\mathbb{M}(\mu)=\sum_{i=1}^J \mathbb{M}(\mu_i)}} \tilde{F}_{p,C}(\mu)= \min_{\substack{\mu\in\mathcal{M}_+(\tilde{\Y})\\\mathbb{M}(\mu)=\sum_{i=1}^J \mathbb{M}(\mu_i)}} {F}_{p,C}(\mu_{\vert \Y})\geq \min_{\substack{\mu\in\msrY\\\mathbb{M}(\mu)\leq \sum_{i=1}^J\mathbb{M}(\mu_i)}}F_{p,C}(\mu).
\end{align*}
Combining both inequalities and using \Cref{lem:mass_cap} then finishes the proof.
\end{proof}

\begin{proof}[Proof of \Cref{lem:borelapplication_prop}]
\textbf{(i)} By definition, the objective value for $\tilde{T}^{J,p}_C(y_1,\ldots,y_J)$ at $\dum$ is equal to $(J-\lvert \BS\rvert)C^p/2$. Thus, $\tilde{T}_C^{J,p}$ outputs $\dum$ if and only if for any $y \in \Y$ it holds
\begin{align*}
\sum_{i=1}^J \tilde{d}^p_C(y_i,y) \geq (J-\lvert \BS\rvert)C^p/2
\end{align*}
which is equivalent to
\begin{align*}
\es{i \not \in \BS}{} \tilde{d}^p_C(y_i,y)\geq (J-2\lvert \BS\rvert)C^p/2.
\end{align*}
In particular, if all inequalities are strict $\dum$ is the unique output for $\tilde{T}^{J,p}_C(y_1,\ldots,y_J)$. Statement \textbf{(ii)} is a direct consequence of (i). For statement \textbf{(iii)} we again use that by definition $\tilde{d}^p_C(y,\dum)=\nicefrac{C^p}{2}$ for any $y\in\Y$ and hence
\begin{align*}
\min_{y\in \Y} \sum_{i=1}^J \tilde{d}^p_C(y_i,y) = \lvert \BS \rvert \frac{C^p}{2}+\min_{y\in \Y} \sum_{i\not\in\BS} \tilde{d}^p_C(y_i,y).
\end{align*}
Proving \textbf{(iv)}, let $C>2^{\nicefrac{1}{p}}\text{diam}(\Y)$, pick points $y_1,\ldots,y_J\in\Y$ and observe that for any $y\in\Y$ it holds that
\begin{align*}
\sum_{i=1}^J \tilde{d}^p_C(y_i,\dum)=J\frac{C^p}{2}>J\text{diam}(\Y)^p\geq \sum_{i=1}^J d^p(y_i,y).
\end{align*}
Thus, $\tilde{T}_C^{J,p}(x_1,\ldots,x_J)\neq \dum$ and since $\lvert \BS\rvert=0$, the claim follows from (iii).
\end{proof}

\subsection{Proofs of \Cref{sec:Theory}}\label{sec:proofs}
\begin{proof}[Proof for \Cref{lem:transportC}]
Suppose that $\pi_C$ is optimal but its induced graph $G(\pi_C)$ contains a path $P=(x_{i_1},\ldots,x_{i_k})$ such that $\mathcal{L}(P)>C^p$. By definition of $G(\pi_C)$ it holds that $\pi_C(x_{i_j},x_{i_{j+1}})>0$ for all $1\leq j \leq k-1$. We define a new transport plan with augmented transport along the path $P$. For this, define $\epsilon\coloneqq \min_{1\leq j \leq k-1}\pi_C(x_{i_j},x_{i_{j+1}})$ and construct the new plan
\begin{align*}
\tilde{\pi}_C(x,x^\prime)= \begin{cases}
\pi_C(x,x^\prime)- \epsilon, &\text{ if } \exists\, 1\leq j \leq k-1,\, x=x_{i_j},\, x^\prime=x_{i_{j+1}}\\
\pi_C(x,x^\prime), &\text{ else.}
\end{cases}
\end{align*}
Compared to $\pi_C$ the transportation cost for $\tilde{\pi_C}$ is reduced by $\epsilon\mathcal{L}(P)$ while the marginal deviation is increased by $\epsilon C^p$. In particular, it holds that
\begin{align*}
&\sum_{x,x^\prime} d^p(x,x^\prime)\tilde{\pi}_C(x,x^\prime)+\frac{C^p}{2}\left(\sum_{x} \mu(x)-\tilde{\pi}_C(x,\X) + \sum_{x^\prime} \nu(x^\prime)-\tilde{\pi}_C(\X,x^\prime) \right)\\
=&\sum_{x,x^\prime} d^p(x,x^\prime)\pi_C(x,x^\prime)+\frac{C^p}{2}\left(\sum_{x} \mu(x)-\pi_C(x,\X) + \sum_{x^\prime} \nu(x^\prime)-\pi_C(\X,x^\prime) \right)\\  &\quad +\epsilon \left(C^p-\mathcal{L}(P) \right).
\end{align*}
As $\epsilon>0$ and $\mathcal{L}(P)>C^p$ this contradicts the optimality for $\pi_C$. Consequently, any path $P$ in the induced graph $G(\pi_C)$ necessarily has path length at most $C^p$. If $d(x,x^\prime)>C$ this implies that $d^p(x,x^\prime)>C^p$ and hence by the statement on induced graphs that $\pi_C(x,x^\prime)=0$.
\end{proof}
\begin{proof}[Proof for \Cref{prop:KRdistance}]
We first establish the metric properties \textbf{(i)}. It is straightforward to show $\text{KR}_{p,C}(\mu,\nu)=0$ if and only if $\mu=\nu$ and that $\text{KR}_{p,C}$ is symmetric. For the triangle inequality let $\mu,\nu,\tau\in\msrX$ and choose $B\geq \max\{\mathbb{M}(\mu),\mathbb{M}(\nu),\mathbb{M}(\tau)\}$. Then by augmenting the measures accordingly (\Cref{app:lifttoOT}) we find that 
\begin{align*}
\text{KR}_{p,C}(\mu,\nu)&=\left( \tilde{\text{OT}}_{\tilde{d}^p_C}(\tilde{\mu},\tilde{\nu})\right)^{\nicefrac{1}{p}} \\
&\leq   \left( \tilde{\text{OT}}_{\tilde{d}^p_C}(\tilde{\mu},\tilde{\tau})\right)^{\nicefrac{1}{p}} + \left( \tilde{\text{OT}}_{\tilde{d}^p_C}(\tilde{\tau},\tilde{\nu})\right)^{\nicefrac{1}{p}} = \text{KR}_{p,C}(\mu,\tau) + \text{KR}_{p,C}(\tau,\nu)
\end{align*}
where the inequality follows by the triangle inequality for the Wasserstein distance \cite[Theorem 7.3]{villani2003topics}.
Statement $\textbf{(ii)}$ follows from \Cref{lem:transportC} by noting that there exists at least one optimal solution $\pi_C$ equal to zero except on the diagonal for which $\pi_C(x,x)=\mu(x)\wedge \nu(x)$. Plugging into the objective of \eqref{eq:UOT} yields the claim. Additionally, suppose that w.l.o.g. $\mu(x)\geq \nu(x)$ for all $x\in\X$. Then independent to the choice of $C>0$ and $p\geq 1$ the unique optimal solution is to remain all shared mass at its common place and to delete surplus material which is exactly the solution $\pi_C(x,x)=\mu(x)\wedge \nu(x)$ described before. Statement $\textbf{(iii)}$ follows by noting that for $C\geq\max_{x,x^\prime}d(x,x^\prime)$ the dual formulation in \eqref{eq:dualUOT} and in \eqref{eq:dualOT} coincide. \\
Finally, for statement \textbf{(iv)} we note that by construction it holds $\tilde{d}_{C_1}^p(x,y)\leq \tilde{d}_{C_2}^p(x,y)$ for all $x,y\in \tilde{\mathcal{Y}}$. Hence, for any coupling $\pi$ of the augmented measures $\tilde{\mu},\tilde{\nu}$ it holds
\begin{align*}
    \es{x,x^\prime\in \tilde{\Y}}{}\tilde{d}_{C_1}^p(x,x^\prime)\pi(x,x^\prime)\leq \es{x,x^\prime\in \tilde{\Y}}{}\tilde{d}_{C_2}^p(x,x^\prime)\pi(x,x^\prime).
\end{align*}
Taking the minimum over all couplings of $\tilde{\mu}$ and $\tilde{\nu}$ on both sides completes the proof.
\end{proof}

\subsubsection{Proof for \Cref{thm:KRultrametric}}
\label{app:tree}
Using the lift to the OT problem, we can now start to prove the closed formula on ultra-metric trees. For this, consider an ultrametric tree $\mathcal{T}$ with height function $h\colon V\to \mathbb{R}_+$ and define its \emph{$p$-height transformed tree} denoted $\mathcal{T}_p\coloneqq \mathcal{T}$ as the same tree but with height function $h_p(\textsf{v})=2^{p-1}h(\textsf{v})^p$. An illustration is given in \Cref{fig:explicitformula}. Notice that by monotonicity $\mathcal{T}_p$ is again an ultrametric tree.
\begin{lemma}\label{lem:transformedtree}
Let $\mathcal{T}$ be an ultrametric tree with height function $h\colon V\to \mathbb{R}_+$ and consider its $p$-height transformed tree $\mathcal{T}_p$. Then it holds that
\begin{align*}
d_\mathcal{T}^p(\textsf{v},\textsf{w})=d_{\mathcal{T}_p}(\textsf{v},\textsf{w})
\end{align*}
for all leaf nodes $\textsf{v},\textsf{w}\in L\subset V$.
\end{lemma}

\begin{proof}
Let $\textsf{v},\textsf{w}\in L$ be two leaf nodes in the ultrametric tree $\mathcal{T}$ with height function $h$ and let \textsf{a} be their common ancestor\footnote{If $\textsf{v},\textsf{w}\in L$ are leaf nodes their common ancestor is defined as the node included in the path from $\textsf{v}$ to $\textsf{w}$ closest to the root.}. Since paths between any two vertices are unique and all leaf nodes have the same distance to the root, it holds that
\begin{align*}
d_\mathcal{T}(\textsf{v},\textsf{a})-d_\mathcal{T}(\textsf{w},\textsf{a})&=d_\mathcal{T}(\textsf{v},\textsf{a})+d_\mathcal{T}(\textsf{a},\textsf{r})-d_\mathcal{T}(\textsf{a},\textsf{r})-d_\mathcal{T}(\textsf{w},\textsf{a})\\
&=d_\mathcal{T}(\textsf{v},\textsf{r})-d_\mathcal{T}(\textsf{w},\textsf{r})=0.
\end{align*}
Hence, 
\begin{align*}
\left(d_\mathcal{T}(\textsf{v},\textsf{w})\right)^p=\left(d_\mathcal{T}(\textsf{v},\textsf{a})+d_\mathcal{T}(\textsf{w},\textsf{a})\right)^p=2^p(h(\textsf{a})-h(\textsf{v}))^p=2^ph(\textsf{a})^p,
\end{align*}
where we use that $h(\textsf{v})=0$. Repeating the argument for the ultrametric tree $\mathcal{T}_p$ we conclude that $d_{\mathcal{T}_p}(\textsf{v},\textsf{w})=2d_{\mathcal{T}_p}(\textsf{v},\textsf{a})=2^ph(\textsf{a})^p$.
\end{proof}

Equipped with this result we are now able to prove the closed formula from \Cref{thm:KRultrametric}.

\begin{proof}[Proof for \Cref{thm:KRultrametric}]
Let $\text{KR}^p_{p,C}(\mu,\nu)=\text{UOT}_{p,C}(\mu,\nu)$ refer to UOT w.r.t. the distance on $\mathcal{T}$, which only depends on the distance between individual leaf nodes. Considering the $p$-th height transformed tree $\mathcal{T}_p$ and applying \Cref{lem:transformedtree} we conclude that 

\begin{align*}
\text{KR}_{p,C}^p(\mu^L,\nu^L)=&
\min_{\pi\in \mathcal{M}_+(L\times L)} \sum\limits_{\textsf{v},\textsf{v}^\prime \in L}d_{\mathcal{T}_p}(\textsf{v},\textsf{v}^\prime)\pi(\textsf{v},\textsf{v}^\prime)\\
&+\frac{C^p}{2}\left(\sum\limits_{\textsf{v}\in L}\left(\mu^L(\textsf{v})-\pi(\textsf{v},L)\right)+\sum\limits_{\textsf{v}^\prime\in L}\left(\nu^L(\textsf{v}^\prime)-\pi(L,\textsf{v}^\prime)\right)\right) \notag \\[2ex]
\textrm{s.t. }
&\pi(\textsf{v},L)  \leq  \mu^L(\textsf{v})\, , \quad \forall\, \textsf{v}\in L, \\[2ex]
&\pi(L,\textsf{v}^\prime)  \leq  \nu^L(\textsf{v}^\prime)\, , \quad \forall\, \textsf{v}^\prime \in L. \notag 
\end{align*}

The linear optimization problem can be decomposed on several subtrees. For this recall that by \Cref{lem:transportC} (i) there exists an optimal solution such that mass transportation is only considered on metric scales between two leaf nodes $\textsf{v},\textsf{v}^\prime\in L$ such that $d_{\mathcal{T}_p}(\textsf{v},\textsf{v}^\prime)\leq C^p$. If $\textsf{v}_0$ is the common ancestor of $\textsf{v},\textsf{v}^\prime$ then by the ultrametric tree properties of $\mathcal{T}$ (see also the proof of \Cref{lem:transformedtree}) the inequality $d_{\mathcal{T}_p}(\textsf{v},\textsf{v}^\prime)\leq C^p$ is equivalent to the height function
$h(\textsf{v}_0)\leq \frac{C}{2}$. Consider the set $\mathcal{R}(C)$ in \eqref{eq:subtreeroot} and for each $\textsf{v}\in\mathcal{R}(C)$ define subtrees $\mathcal{C}(\textsf{v})$ consisting of the children of $\textsf{v}$ and the subset of corresponding edges. By construction if $\textsf{v}_i,\textsf{v}_j\in \mathcal{R}(C)$ with $\textsf{v}_i \neq \textsf{v}_j$ then the subtrees are disjoint $\mathcal{C}(\textsf{v}_i)\cap\mathcal{C}(\textsf{v}_j)=\emptyset$ (\Cref{fig:explicitformula} (a) for an illustration). In particular, the linear optimization problem $\text{KR}_{p,C}^p(\mu^L,\nu^L)$ is decomposed on each individual subtree $\mathcal{C}(\textsf{v})$ for each $\textsf{v}\in\mathcal{R}(C)$. The distance on individual subtrees is set to be the $p$-th height transformed tree distance $d_{\mathcal{T}_p}$ which exactly captures the pairwise $p$-th power distance between leaf nodes belonging to the same subtree (\Cref{lem:transformedtree}). For an element $\textsf{v}\in\mathcal{R}(C)$ consider its subtree $\mathcal{C}(\textsf{v})$ with distance $d_{\mathcal{T}_p}$. By definition the maximal distance between its leaf nodes is bounded by $\nicefrac{C^p}{2}$. We augment the subtree $\mathcal{C}(\textsf{v})$ with a dummy node $\tilde{\textsf{v}}$ and introduce an edge $e=(\textsf{v},\tilde{\textsf{v}})$ with edge weight $\frac{C^p}{2}-2^{p-1}h(\textsf{v})^p$ (\Cref{fig:explicitformula} (b) for an illustration). Denote the augmented tree by $\tilde{\mathcal{C}}(\textsf{v})$. Considering the measures $\mu^L,\nu^L$ restricted to $\mathcal{C}(\textsf{v})$ we augment $\mu^L$ adding mass $\left(\mu^L(\mathcal{C}(\textsf{v}))-\nu^L(\mathcal{C}(\textsf{v}))\right)_+$ at $\tilde{\textsf{v}}$ and vice versa augment $\nu^L$ adding mass $\left(\nu^L(\mathcal{C}(\textsf{v}))-\mu^L(\mathcal{C}(\textsf{v}))\right)_+$ at $\tilde{\textsf{v}}$. This construction defines an equivalent OT problem on $\tilde{\mathcal{C}}(\textsf{v})$ \citep{guittet2002extended}. Hence, applying the closed formula for OT on general metric trees \cite[p.575]{evans2012phylogenetic} yields
\begin{align*}
2^{p-1}\sum_{\textsf{w}\in \mathcal{C}(\textsf{v})\setminus \lbrace \textsf{v} \rbrace}  \Big(\left( h(par(\textsf{w}))^p-h(\textsf{w})^p \right) \left\vert \mu^L(\mathcal{C}(\textsf{w}))-\nu^L(\mathcal{C}(\textsf{w}))\right\vert \Big) \\
+\left(\frac{C^p}{2}-h(\textsf{v})\right) \left\vert \mu^L(\mathcal{C}(\textsf{v}))-\nu^L(\mathcal{C}(\textsf{v}))\right\vert.
\end{align*}
Summing over all subtrees indexed by the set $\mathcal{R}(C)$ finishes the proof.
\end{proof}

\subsubsection{Proofs for the Barycenter}

\begin{proof}[Proof for \Cref{thm:bary_prop}]
\textcolor{mygreen}{\textbf{(ii)} Let $\mu$ be a $(p,C)$-barycenter and $\tilde{\mu}$ its augmented counterpart. Then, by \Cref{prop:multimarginal} there exists an optimal multi-coupling $\pi$, such that it holds $\mu=\tilde{\mu}_{\vert \mathcal{Y}}=(\tilde{T}_C^{J,p}\# \pi)_{\lvert \mathcal{Y}}$. Hence, for each $y\in \text{supp}(\tilde{\mu})$ there exists $K_y\geq 1$ and $K_y$ $J$-tupels $(x^1_{y,1},\dots,x^J_{y,1}),\dots,(x^1_{y,K_y},\dots,x^J_{y,K_y})$ such that for $k=1,\dots ,K_y$ it holds
\begin{align*}
    y=\tilde{T}^{J,p}_C(x^1_{y,k},\dots,x^J_{y,k})
\end{align*}
and $\mu(y)=\sum_{k=1}^{K_y}a^y_k$, where $a_k^y=\pi(x^1_{y,k},\dots,x^J_{y,k})$. For $i=1,\dots,J$ define $\tilde{\pi}_i\in \mathcal{M}_+(\Y\times \Y)$ by $\tilde{\pi}_i(y,x^i_{y,k})=a_k^y$ for all $y\in \text{supp}(\tilde{\mu})$, $i=1,\dots,J$ and $k=1,\dots,K_y$. Set $\tilde{\pi}_i$ to be zero everywhere else for $i=1,\dots,J$. By construction,  $\tilde{\pi}_i$ defines an OT plan between $\tilde{\mu}$ and $\tilde{\mu}_i$ for $i=1,\dots,J$. It holds
\begin{align*}
 \frac{1}{J} \sum_{i=1}^J \tilde{OT}_{\tilde{d}^p_C}(\tilde{\mu},\tilde{\mu_i})&=   \es{x\in \text{supp}(\pi)}{} c_{p,C}(x)\pi(x) 
\\
&=\es{x\in \text{supp}(\pi)}{} \frac{1}{J} \sum_{i=1}^J \tilde{d}_C^p\left(x_i,\tilde{T}_C^{J,p}(x)\right) \pi(x)\\
&=\frac{1}{J}\sum_{y\in \text{supp}(\tilde{\mu})}\sum_{i=1}^J\sum_{k=1}^{K_y}\tilde{d}^p_C(x_{y,k}^i,y)a_k^y\\
&=\frac{1}{J}\sum_{i=1}^J\sum_{y\in \text{supp}(\tilde{\mu})}\sum_{k=1}^{K_y}\tilde{d}^p_C(x_{y,k}^i,y)\tilde{\pi}^i(y,x^i_{y,k}),
\end{align*}
where the first equality follows from \Cref{prop:multimarginal} and the third and fourth by construction. Since $\tilde{\pi}_i$ is an OT plan between $\tilde{\mu}$ and $\tilde{\mu}_i$ it holds for all $i=1,\dots,J$ that
\begin{align*}
\sum_{y\in \text{supp}(\tilde{\mu})}\sum_{k=1}^{K_y}\tilde{d}^p_C(x_{y,k}^i,y)\tilde{\pi}_i(y,x^i_{y,k}) \geq \tilde{OT}_{p,C}^p(\tilde{\mu},\tilde{\mu}_i).
\end{align*}
Thus, it follows together with the previous equations that
\begin{align*}
    \tilde{OT}_{p,C}^p(\tilde{\mu},\tilde{\mu}_i)=\sum_{y\in \text{supp}(\tilde{\mu})}\sum_{k=1}^{K_y}\tilde{d}^p_C(x_{y,k}^i,y)\tilde{\pi}_i(y,x^i_{y,k}),
\end{align*}
i.e. $\tilde{\pi}_i$ is optimal. \Cref{lem:borelapplication_prop} now yields the first part of the statement. \\
For the second part assume that for any $(x_1,\dots,x_L)\in \Y^L$ it holds that $T^{L,p}(x_1,\dots,x_L)=T^{L,p}(y_1,x_2,\dots,x_L)$ is equivalent to $x_1=y_1$. Let $y\in \text{supp}(\tilde{\mu})$ and consider OT plans $\tilde{\pi}^1,\dots,\tilde{\pi}^J$ between $\tilde{\mu}$ and $\tilde{\mu}_i$, respectively. For $i=1,\dots,J$ consider $x^i$ such that $\tilde{\pi}_i(y,x^i)=a_i>0$. Assume that it holds $y\neq \tilde{T}_C^{J,p}(x_1,\dots,x_J)$. Denote the minimum of the $a_i$ as $a_0=\min_{i=1,\dots,J} a_i$. By construction, it follows that
\begin{align*}
    \tilde{F}_{p,C}(\tilde{\mu}-a_0 \delta_y+a_0 \delta_{\tilde{T}_C^{J,p}(x_1,\dots,x_J)})<F_{p,C}(\tilde{\mu}),
\end{align*}
which is a contradiction to $\tilde{\mu}$ being a barycenter of $\tilde{\mu}_1,\dots,\tilde{\mu}_J$. Thus, it holds $y=\tilde{T}_C^{J,p}(x_1,\dots,x_J)$. Now, assume w.l.o.g. there exists $x_1,z_1\in \Y$, such that it holds $\pi^1(y,x_1)>0$ and $\pi^1(y,z_1)>0$. However, by the previous argument this implies 
\begin{align*}
    \tilde{T}_C^{J,p}(x_1,\dots,x_J)=y=\tilde{T}_C^{J,p}(z_1,\dots,x_J).
\end{align*}
By assumption this is equivalent to $x_1=z_1$, thus it holds for all $x,y\in \Y$ and $i=1,\dots,J$ that $\pi^i(y,x)\in \{0,\mu(y) \}$.}\\
\textbf{(i)} By \Cref{prop:multimarginal} the objective value of the balanced multi-marginal and $(p,C)$-barycenter problem coincide and a $(p,C)$-barycenter is obtained as the push-forward of an optimal balanced multi-coupling under the map $\tilde{T}_C^{J,p}$ restricted to $\Y$. By construction and \Cref{cor:centroid_sets} any such measure is supported in $\mathcal{C}_{\text{KR}}(J,p,C)$. Thus, there always exists a $(p,C)$-barycenter whose support is restricted to $\mathcal{C}_{\text{KR}}(J,p,C)$ and the minimum over ${Y}$ and $\mathcal{C}_{\text{KR}}(J,p,C)$ coincide.\\
The second part is similar and we let $\tilde{\mu}$ be any $p$-Wasserstein barycenter. Then by \Cref{prop:multimarginal}, there exists a multi-coupling of $\tilde{\mu}_1,\dots ,\tilde{\mu}_J$, such that $\tilde{\mu}=\tilde{T}_C^{J,p}\# \tilde{\pi}$. Since any such push-forward measure can only have support in $\mathcal{C}_{\text{KR}}(J,p,C)\cup \{\dum\}$, it holds for $\mu=\tilde{\mu}_{\vert \Y}$ that $\text{supp}(\mu)\subset \mathcal{C}_{\text{KR}}(J,p,C)$. It remains to show the upper bound on the total mass. By the equivalence to the multi-marginal problem and by \Cref{lem:borelapplication_prop} (ii) any $(p,C)$-barycenter $\mu$ cannot have mass on a point which is constructed from a set of points $(x_1,\dots ,x_J)$ for which $2\lvert \BS(x_1,\dots ,x_J)\rvert \geq J$. Additionally, by part $(ii)$ we know that there exists UOT plans, such that the mass of each $(p,C)$-barycenter support point is fully transported to points it is constructed from. Let $(a_1,\dots , a_K)$ be the weight vector of the support points of the $(p,C)$-barycenter, then it holds that
\begin{align*}
\sum_{i=1}^J \mathbb{M}(\mu_i)- \lceil J/2 \rceil \es{k=1}{K} a_k \geq 0 ,
\end{align*}
since by the previous argument and \Cref{lem:borelapplication_prop}, any $(p,C)$-barycenter support point $x_k$ reduces the maximum available mass by at least $\lceil J/2 \rceil  a_k$ and by Lemma~\ref{lem:mass_cap}, the total mass of the $(p,C)$-barycenter is bounded by the sum of the total masses of the $\mu_i$. Therefore it holds that
\begin{align*}
 \mathbb{M}(\mu)  = \es{k=1}{K} a_k \leq \lceil J/2 \rceil^{-1} \sum_{i=1}^J \mathbb{M}(\mu_i) \leq \frac{2}{J}\sum_{i=1}^J \mathbb{M}(\mu_i).
\end{align*}
\textbf{(iii)} The multi-marginal problem between $\tilde{\mu}_1,\dots ,\tilde{\mu}_J$ is a balanced problem, thus we can pose this as a linear program with a total of $\prod_{i=1}^JM_i$ variables and $\sum_{i=1}^N M_i+J$ constraints. As all measures have the same total mass, we can drop one arbitrary marginal constraint for each measure besides the first. Thus, the rank of the constraint matrix in the corresponding constraint is bounded by $\sum_{i=1}^N M_i +1$. Hence, each basic feasible solution of the linear program has at most $\sum_{i=1}^N M_i + 1$ non-zero entries (see \citep{luenberger1984linear} for details). Let $\pi$ be such a solution. By \Cref{prop:multimarginal} the measure $\tilde{\mu}=\tilde{T}_C^{J,p}\# \pi$ is a $p$-Wasserstein barycenter and by construction it has at most $\sum_{i=1}^N M_i +1$ support points. Due to the upper bound on the total mass of the $(p,C)$-barycenter in property $(i)$, we can guarantee that there is non-zero mass at $\dum$ for $J>2$, hence in this case, restricting the measure to $\Y$ reduces the support size by one. For $J=2$, we note that the multi-marginal problem is just the augmented UOT problem. By construction we either have a point $x$ in the support of one of the two measures, such that there is transport between $x$ and $\dum$ or both measures have equal mass at $\dum$ and it is optimal to leave this mass in place. In the first case, we have mass at $\tilde{T}^{J,p}_C(x,\dum)=\dum$, thus the support size can be reduced by one and in the second the problem is equivalent to the OT problem and thus the barycenter has at most $M_1+M_2-1$ support points. Finally, by property $(i)$ the support of any $(p,C)$-barycenter is contained in $\mathcal{C}_{KR}(J,p,C)$, thus the cardinality of this set also provides a trivial upper bound on the support size of any $(p,C)$-barycenter. Taking the minimum over both quantities, we conclude 
\begin{align*}
\lvert \text{supp}(\mu)  \rvert \leq \min \left\{ \lvert \mathcal{C}_{KR}(J,p,C) \rvert , \sum_{i=1}^J M_i \right\}.
\end{align*}
\textbf{(iv)} For any $\mu \in \msrY$, it holds
\begin{align*}
    F_{p,C_1}^p(\mu) = \frac{1}{J}\es{i=1}{J}\text{KR}_{p,C_1}^p(\mu,\mu_i)\leq \frac{1}{J}\es{i=1}{J}\text{KR}_{p,C_2}^p(\mu,\mu_i) = F_{p,C_2}^p(\mu),
\end{align*}
where the inequality follows from \Cref{prop:KRdistance} (iv). Taking the infimum over all measures in $\msrY$ on both sides completes the proof.\\
\textbf{(v)} Let $C\leq d_{\min}^\prime$, then by \Cref{prop:KRdistance} (ii) it holds
\begin{align*}
\uamin{\mu \in \msrY} F_{p,C}(\mu)&= \ \uamin{\mu \in \msrY}  \frac{C^p}{2J} \sum_{i=1}^J TV(\mu,\mu_i) \\
&= \uamin {a \in \mathbb{R}_+^K} \frac{C^p}{2J} \sum_{i=1}^J \es{k=1}{K} \lvert a_k - a^i_k \rvert \\
 &=\uamin {a \in \mathbb{R}_+^K} \frac{C^p}{2J} \es{k=1}{K}  \sum_{i=1}^J \lvert a_k - a^i_k \rvert,
\end{align*} 
where the change in the $\argmin$ in the second line follows from the fact that the total variation can only increase if we place mass outside of the support of the measures. Thus it suffices to consider measures supported on the union of the supports. Now, we note that the $K$ summands are independent to each other, thus we can minimise them separately. Hence, for the $k$-th entry of $a$ it holds that
\begin{align*}
a_k \in \uamin{a \in \mathbb{R}_+} \esn \lvert a_k - a^i_k \rvert= med(a_k^1,\dots ,a_k^J)
\end{align*}
which yields the claim. \\

\bigskip
\textbf{(vi)} 
Let $\mathbb{M}_i=\mathbb{M}(\mu_i)$ for $i=1,\dots ,J$ and set $\mathbb{M}_0=0$. Assume that $J$ is odd. Let $\mu$ be a $(p,C)$-barycenter of $\mu_1,\dots ,\mu_J$ with $\mu(\mathcal{Y})\in [\mathbb{M}_{k-1},\mathbb{M}_k]$. In particular, $\mu$ fulfills the non-mass-splitting property in (ii). Let $a\in (0,\mathbb{M}_k-\mathbb{M}(\mu)]$ and $\tilde{\mu}$ the augmented measure for $\mu$. By construction, we can find support points $x_k,\dots ,x_J\neq \dum$ of the augmented measures $\tilde{\mu}_k,\dots ,\tilde{\mu}_J$ from which w.l.o.g. mass $a$ is transported to $\dum$ in $\mu$. If one of the points has mass smaller $a$, we can just replace $a$ with the minimum of the masses of the points and repeat the argument until we have considered a total mass of $a$. 
Set $x_0=\tilde{T}_C^{J,p}(\dum,\dots, \dum,x_k,\dots ,x_J)$ and notice that if $x_0=\dum$, we do not change the objective function in the augmented problem (\Cref{lem:add_mass}) by adding this point which means w.l.o.g. $x_0\neq \dum$. In this case, we have 
\begin{align*}
x_0=\uamin{x\in \mathcal{Y}} \sum_{i=k}^{J} \tilde{d}_C^p(x_i,x).
\end{align*}
Now, the objective cost of not having mass $a$ at $x_0$ is $aC^p(J-k)/2$, while the cost of adding $a\delta_{x_0}$ to $\mu$ is equal to $a(kC^p/2 + \sum_{i=k}^{J}\tilde{d}_C^p(x_i,x_0))$. Hence, adding the point improves the value of the Fr\'echet functional, if
\begin{align*}
\es{i=k}{J}\tilde{d}_C^p(x_i,x_0)\leq  C^p(J-2k)/2.
\end{align*}
For $2k>J$, the right hand side will always be negative, so we can not improve. Thus, we assume $2k<J$. By assumption it holds $C\geq J^{\frac{1}{p}} \text{diam}(\mathcal{Z})$. Hence,
\begin{align*}
 &\frac{C^p}{2}\geq \frac{J}{2}  diam(\mathcal{Z})^p \\
 \Leftrightarrow \quad &\frac{C^p}{2\text{diam}(\mathcal{Z})^p}\geq \frac{J}{2}\geq \frac{J-k}{J-2k} \\
 \Leftrightarrow  \quad &C^p(J-2k)/2\geq \text{diam}(\mathcal{Z})^p (J-k)\geq \es{i=k}{J}\tilde{d}_C^p(x_i,x_0).
\end{align*}
Therefore, for $2k<J$ the objective value of $\mu$ can always be improved by increasing its mass by $a$, as long as $k < \lceil J/2 \rceil$. Thus, since $\mu$ is a barycenter it holds $\mathbb{M}(\mu)\geq \mathbb{M}\mu_{\lceil J/2 \rceil}$.\\
An analog, converse argument yields that if $k>J/2$, we can always improve the objective value of $\mu$, since removing and then re-adding any mass to $\mu$ increases the objective value by the previous argument. Hence, it holds $\mathbb{M}(\mu)= \mathbb{M}\mu_{\lceil J/2 \rceil}$.\\
Now, assume $J$ is even. For $2k\neq J$ nothing in the previous argument changes. However, for $2k=J$ (note that this can only hold now that $J$ is even), the right hand side is zero, however, if all the $x_i$ for $i=k,\dots, J$, are identical to $x_0$ (in particular, there exists a point contained in the support of at least half of the measures), then the left hand side will also be zero. In this case, the presence of this point does not change the objective value and there are $(p,C)$-barycenters of different total masses. However, we can still always choose to not place mass in such cases, to obtain a $(p,C)$-barycenter of the desired total mass.
\end{proof}

\begin{proof}[Proof for \Cref{lem:cluster}]
It suffices to show that there is no centroid point, which is constructed from points from two or more different sets $B_r$. Assume there is a point $y_0 \in \mathcal{C}_{KR}(J,p,C)$, such that $y_0$ is constructed, among others, from $x_1 \in B_r$ and $x_2\in B_s$ for $r\neq s$. We distinguish two cases. Assume $y_0 \in B_r$, then it holds $d^p(x_1,y_0)>2^{p-1}C^p\geq C^p$ and $y_0$ would not be in the restricted centroid set. The analogue argument holds for $y_0 \in B_s$. Now, assume $y_0$ is neither in $B_r$ nor $B_S$. Since $d(B_r,B_S)>2^{1/p}C$, it holds either $d^p(B_r,y_0)>C^p$ or $d^p(B_s,y_0)>C^p$. Thus, we obtain another contradiction to $y_0\in  \mathcal{C}_{KR}(J,p,C)$. Hence, $ \mathcal{C}_{KR}(J,p,C)$ only contains centroids constructed from points within one $B_r$ and by convexity of the $B_r$, any centroid point constructed from points within $B_r$ is again in $B_r$. \Cref{thm:bary_prop} (ii) yields that there will always be an optimal solution which only transports within each $B_r$, thus the $R$ problems are in fact independent and we can separate them without changing the objective value. 
\end{proof}

\end{appendix}

\end{document}